\documentclass[leqno]{article}
\usepackage[french,english]{babel}
\usepackage[utf8]{inputenc}
\usepackage[T1]{fontenc}
\usepackage{amsmath}
\usepackage{amssymb}
\usepackage{amsfonts}
\usepackage{enumerate}
\usepackage{vmargin}
\usepackage[all]{xy}
\usepackage{mathrsfs}
\usepackage{mathtools}
\usepackage{lmodern}
\usepackage{slashed}
\usepackage{tikz}
\usetikzlibrary{decorations.pathreplacing}
\usepackage[colorlinks=true,linkcolor=blue,pagebackref=true,hyperfootnotes=false]{hyperref}
\setmarginsrb{3cm}{3cm}{3.5cm}{3cm}{0cm}{0cm}{1.5cm}{3cm}
\usepackage{comment}
\usepackage{accents}
\newlength{\dhatheight}

\newcommand{\Cbb}{\ensuremath{\mathbb{C}}}
\newcommand{\C}{\mathrm{C}}

\newcommand{\N}{\ensuremath{\mathbb{N}}}
\newcommand{\B}{\mathrm{B}}

\let\L\relax
\newcommand{\L}{\mathrm{L}}

\newcommand{\M}{\mathrm{M}}

\newcommand{\diag}{\mathrm{diag}}

\newcommand{\dist}{\mathrm{dist}}

\newcommand{\HS}{\mathrm{HS}}

\newcommand{\Z}{\ensuremath{\mathbb{Z}}}

\newcommand{\R}{\ensuremath{\mathbb{R}}}
\newcommand{\T}{\ensuremath{\mathbb{T}}}

\let\j\relax
\newcommand{\j}{\frak{j}}
\let\S\relax
\newcommand{\S}{\mathrm{S}}

\newcommand{\Id}{\mathrm{Id}}

\newcommand{\VN}{\mathrm{VN}}

\newcommand{\la}{\langle}
\newcommand{\ra}{\rangle}

\renewcommand{\leq}{\ensuremath{\leqslant}}
\renewcommand{\geq}{\ensuremath{\geqslant}}
\newcommand{\qed}{\unskip\nobreak\hfill\penalty50\hskip1em\null\nobreak\hfill\hbox{\vrule height6pt width6pt depth0pt}}
\newcommand{\bnorm}[1]{ \big\| #1  \big\|}

\newcommand{\norm}[1]{\left\Vert#1\right\Vert}

\newcommand{\co}{\colon}

\newcommand{\ot}{\otimes}
\newcommand{\ovl}{\overline}
\newcommand{\otvn}{\ovl\ot}

\newcommand{\cb}{\mathrm{cb}}

\let\i\relax
\newcommand{\i}{\mathrm{i}}
\let\j\relax
\newcommand{\j}{\mathrm{j}}
\let\k\relax
\newcommand{\k}{\mathrm{k}}

\newcommand{\ov}{\overset}

\newcommand{\epsi}{\varepsilon}
\renewcommand{\d}{\mathop{}\mathopen{}\mathrm{d}}
\newcommand{\e}{\mathrm{e}}
\renewcommand{\d}{\mathop{}\mathopen{}\mathrm{d}}

\DeclareMathOperator{\Span}{span}

\DeclareMathOperator{\supp}{supp}

\DeclareMathOperator{\Tr}{Tr}

\let\Re\relax
\DeclareMathOperator{\Re}{Re}

\newtheorem{thm}{Theorem}[section]
\newtheorem{defi}[thm]{Definition}
\newtheorem{prop}[thm]{Proposition}

\newtheorem{cor}[thm]{Corollary}
\newtheorem{lemma}[thm]{Lemma}

\newtheorem{remark}[thm]{Remark}
\newtheorem{example}[thm]{Example}

\newenvironment{proof}[1][]{\noindent {\it Proof #1} : }{\hbox{~}\qed
\smallskip
}
\usepackage{tocloft}
\numberwithin{equation}{section}
\usepackage[nottoc,notlot,notlof]{tocbibind}
\let\OLDthebibliography\thebibliography
\renewcommand\thebibliography[1]{
  \OLDthebibliography{#1}
  \setlength{\parskip}{0pt}
  \setlength{\itemsep}{0pt plus 0.3ex}
}
\newcommand\reallywidehat[1]{\arraycolsep=0pt\relax
\begin{array}{c}
\stretchto{
  \scaleto{
    \scalerel*[\widthof{\ensuremath{#1}}]{\kern-.5pt\bigwedge\kern-.5pt}
    {\rule[-\textheight/2]{1ex}{\textheight}}
  }{\textheight}
}{0.5ex}\\
#1\\
\rule{-1ex}{0ex}
\end{array}
}
\newcommand{\FF}{\mathbb{F}}
\begin{document}
\selectlanguage{english}

\title{\bfseries{Radial Schur multipliers on trees and quantitative Fourier--Schur separation}}
\date{}
\author{\bfseries{C\'edric Arhancet}}
\maketitle

\begin{abstract}
We study ordinary and completely bounded norms of radial Schur multipliers on trees. On the homogeneous tree of countably infinite degree, we show that boundedness implies complete boundedness. The argument realizes arbitrary matrix amplifications of a shifted Toeplitz symbol inside rectangular corners of the original multiplier. It also applies to some locally finite trees of unbounded degree. Consequently, the completely bounded radial multiplier theorem of Gonz\'alez-P\'erez, Parcet, P\'erez Garc\'ia and Wang has no general counterpart with ordinary norms. We also determine the optimal order of the comparison constant for positive definite even symbols. For each $1 < p < \infty$ with $p \neq 2$, we construct radial symbols on finitely generated nonabelian free groups, supported in the ball of radius $N$, whose completely bounded Herz--Schur norms remain bounded as $N \to \infty$, whereas their ordinary Fourier multiplier norms grow at least like $N^{|1-\frac{2}{p}|}$. For each fixed $0 < \vartheta < \pi$, we also determine the exact order $t^{-|1-\frac{2}{p}|}$ of both the ordinary and completely bounded Fourier multiplier norms of the damped radial phases $s \mapsto \e^{(-t+\i\vartheta)|s|}$ as $t \to 0$. Finally, we compute the local amenability modulus of free balls, we bound the Fourier--Schur ratio of free balls from above and below in terms of this modulus, and we show that its restriction to radial symbols grows polynomially in the radius.
\end{abstract}

\makeatletter
\renewcommand{\@makefntext}[1]{#1}
\makeatother
{\hypersetup{linkcolor=blue}\tableofcontents}

\section{Introduction}
\label{sec:Introduction}

The comparison between Fourier and Schur multipliers is a basic problem in noncommutative harmonic analysis. For a discrete group $G$ and a bounded function $a \co G \to \Cbb$, the algebraic Fourier multiplier $M_a \co \mathbb{C}[G] \to \mathbb{C}[G]$ is defined on the group algebra by
\[
M_a\bigg(\sum_{s \in G} c_s\lambda_s\bigg)
\ov{\mathrm{def}}{=} \sum_{s \in G} a(s)c_s\lambda_s,
\]
where both sums are finite. We use the same notation for its bounded extension to the noncommutative $\L^p$-space $\L^p(\VN(G))$, when such an extension exists. When it is well-defined, the Herz--Schur multiplier
\begin{equation}
\label{eq:herz-intro}
M_a^{\HS}([x_{s,t}]) 
\ov{\mathrm{def}}{=} [a(st^{-1})x_{s,t}], \quad x \in S^p_G
\end{equation}
acts on the Schatten space $S^p_G \ov{\mathrm{def}}{=} S^p(\ell^2(G))$. The completely bounded norms of these operators coincide when $G$ is amenable, by \cite[Theorem 2.7 p.~1171]{NeR11}. The corresponding result for locally compact groups is due to Caspers and de la Salle \cite[Corollary 5.3]{CaS15}. At $p = \infty$, the comparison is isometric for completely bounded norms for arbitrary groups by the theorem of Bo\.{z}ejko and Fendler \cite{BoF84}, see also \cite{Spr04}.

Gonz\'alez-P\'erez, Parcet, P\'erez Garc\'ia and Wang \cite{GPPW26} proved a radial Schur multiplier theorem for trees and used it to construct counterexamples to Fourier--Schur transference on nonamenable groups. Let $X$ be a countable tree, endowed with its graph distance $\dist_X$. For a complex function $m \co \N_0 \to \Cbb$, we write $M_{m \circ \dist_X}$ for the Schur multiplier with symbol $X \times X \to \mathbb{C}$, $(x,y) \mapsto m(\dist_X(x,y))$, and we denote by $\widetilde{m}(n) \ov{\mathrm{def}}{=} m(|n|)$ its even extension to $\Z$. If $1 < p < \infty$, their theorem gives
\begin{equation}
\label{eq:radial-comparison-intro}
\norm{M_{m \circ \dist_X}}_{\cb,S^p_{X} \to S^p_{X}} 
\lesssim (p^*)^2 \norm{M_{\widetilde m}}_{\cb,\L^p(\T) \to \L^p(\T)},
\quad p^* \ov{\mathrm{def}}{ = } \max\left\{p,\frac{p}{p-1}\right\}.
\end{equation}
If $X$ has infinite diameter, then $\norm{M_{\widetilde m}}_{\cb,\L^p(\T) \to \L^p(\T)} \leq \norm{M_{m \circ \dist_X}}_{\cb,S^p_{X} \to S^p_{X}}$. The ordinary lower estimate is also valid. However, the proof of the upper estimate uses the extension of $M_{\widetilde m}$ to $S^p$-valued functions.

The distinction between ordinary and completely bounded norms is relevant to a question of Pisier \cite[Conjecture 8.1.12 p.~96]{Pis98}: does there exist a bounded Schur multiplier on $S^p$ which is not completely bounded, for $1 < p < \infty$ and $p \neq 2$? A finite-dimensional strict inequality between the two norms does not suffice. Such finite-dimensional examples are now available for all these exponents by \cite{HST26}, see also \cite{Arh26} for a simpler example in the case $p=4$. On the Fourier side, bounded multipliers which are not completely bounded exist on $\L^p(\T)$, see \cite[Proposition 8.1.3 p.~91]{Pis98} and \cite[Theorem 1.5 p.~18]{Arh12}.

Our first result identifies an obstruction to replacing both completely bounded norms in \eqref{eq:radial-comparison-intro} by ordinary norms. Let $X_\infty$ denote the countable homogeneous tree in which each vertex has countably infinitely many neighbours. We write $r_{\cb,p}$ for the completely bounded norm of the Riesz projection $P_+ \co \L^p(\T)\to \L^p(\T)$.

\begin{thm}
\label{thm:intro-automatic}
Let $1 < p < \infty$ and let $m \co \N_0 \to \Cbb$. The following assertions are equivalent:
\begin{enumerate}
\item the Schur multiplier $M_{m \circ \dist_{X_\infty}}$ is bounded on $S^p_{X_\infty}$,
\item the Schur multiplier $M_{m \circ \dist_{X_\infty}}$ is completely bounded on $S^p_{X_\infty}$,
\item the Fourier multiplier $M_{\widetilde m}$ is completely bounded on $\L^p(\T)$.
\end{enumerate}
Moreover, if the first assertion holds, then
\begin{equation}
\label{eq:intro-automatic-estimate}
\norm{M_{\widetilde m}}_{\cb,\L^p(\T) \to \L^p(\T)} 
\leq (2r_{\cb,p}+\sqrt{3})\norm{M_{m \circ \dist_{X_\infty}}}_{S^p_{X_\infty} \to S^p_{X_\infty}}
\end{equation}
and
\begin{equation}
\label{eq:intro-automatic-schur}
\norm{M_{m \circ \dist_{X_\infty}}}_{\cb,S^p_{X_\infty} \to S^p_{X_\infty}} 
\lesssim (p^*)^3\norm{M_{m \circ \dist_{X_\infty}}}_{S^p_{X_\infty} \to S^p_{X_\infty}}.
\end{equation}
\end{thm}

The proof uses only rectangular compressions, transference on $\Z$ and a decomposition of a shifted symbol. The same argument applies to a geometric multiplicity condition stated in Section~\ref{sec:automatic}. In particular, infinite valence at a vertex is not necessary: a locally finite tree may contain the required finite configurations. Theorem~\ref{thm:intro-automatic} does not assert equality of the ordinary and completely bounded norms.

For a discrete group $G$ and a nonempty finite subset $\Omega$ of $G$, define
\begin{equation}
\label{eq:amenability-modulus}
\delta_G(\Omega)
\ov{\mathrm{def}}{=}
\inf_{\norm{\xi}_{\ell^2_G} = 1}
\sup_{s \in \Omega}
\big|1-\la\lambda_s\xi,\xi \ra_{\ell^2_G}\big|.
\end{equation}
The infimum is over unit vectors in the complex Hilbert space $\ell^2_G$. This is the amenability modulus of \cite[Section 4]{GPPW26}. It vanishes on every finite set if $G$ is amenable, and it tends to $1$ along any increasing sequence of finite sets exhausting $G$ if $G$ is nonamenable, see \cite[Section 1]{PRS22} and \cite[Section 4]{GPPW26}. Thus $1-\delta_G(\Omega)$ measures how much amenability survives at the scale $\Omega$. If $r \geq 2$ is an integer, we compute \eqref{eq:amenability-modulus} for the balls $B_N \ov{\mathrm{def}}{=} \{s \in \FF_r : |s| \leq N \}$ of a free group $\FF_r$. For $q \ov{\mathrm{def}}{=} 2r-1$, we prove
\begin{equation}
\label{eq:intro-modulus}
1-\delta_{\FF_r}(B_N) = q^{-\frac{N}{2}}\left(1+\frac{q-1}{q+1}N\right).
\end{equation}
By a classical computation of Cohen \cite{Coh82}, the right-hand side of \eqref{eq:intro-modulus} is the operator norm of the spherical average $\frac{1}{|W_N|}\sum_{s \in W_N}\lambda_s$, where $W_N \ov{\mathrm{def}}{=} \{s \in \FF_r : |s| = N\}$. So \eqref{eq:intro-modulus} can be read as a local and quantitative form of Kesten's amenability criterion \cite{Kes59}. Indeed, for any discrete group $G$ and any probability measure $\mu$ on $\Omega$, averaging the inequality $|1-\la\lambda_s\xi,\xi\ra_{\ell^2_G}| \geq 1-\Re\la\lambda_s\xi,\xi\ra_{\ell^2_G}$ against $\mu$ gives
\begin{equation}
\label{eq:intro-Kesten}
\delta_G(\Omega) 
\geq 1-\norm{\lambda(\mu)}_{\ell^2_G \to \ell^2_G},
\quad 
\lambda(\mu) \ov{\mathrm{def}}{=} \sum_{s \in \Omega}\mu(s)\lambda_s.
\end{equation}
Theorem~\ref{thm:modulus} shows that \eqref{eq:intro-Kesten} is an equality for the balls of $\FF_r$, with $\mu$ the uniform probability measure on the outer sphere $W_N$. In particular, the inner spheres play no role, since $\delta_{\FF_r}(W_N) = \delta_{\FF_r}(B_N)$ by Corollary~\ref{cor:spheres}. Almost optimal vectors are given by the normalized truncations of $s \mapsto q^{-\frac{|s|}{2}}$, which are positive and radial. The variant of the modulus used in \cite{PRS22} coincides with \eqref{eq:amenability-modulus} for discrete groups, since the infimum may be restricted to nonnegative unit vectors, see Lemma~\ref{lem:modulus-positive-vectors}.

This makes the local Fourier--Schur transference of Parcet, Ricard and de la Salle explicit on free balls. For $1 < p < \infty$ and a nonempty finite subset $\Omega$ of a discrete group $G$, following \cite[Section 4]{GPPW26}, we define
\begin{equation}
\label{eq:local-ratio}
\mathcal{R}_{G,p}(\Omega)
\ov{\mathrm{def}}{=}
\sup_{\substack{a \co G \to \mathbb{C},\ a \neq 0\\ \supp a \subset \Omega}}
\frac{\norm{M_a}_{\cb,\L^p(\VN(G)) \to \L^p(\VN(G))}}
{\norm{M_a^\HS}_{\cb,S^p_G \to S^p_G}}.
\end{equation}
Since $\Omega$ is finite, all the multipliers appearing in this supremum are completely bounded. Moreover, testing on matrix units gives $\norm{M_a^\HS}_{\cb,S^p_G \to S^p_G} \geq \norm{a}_{\ell^\infty_G}$, so the denominator is positive whenever $a \neq 0$.

If $p$ is an even integer, combining \cite[Theorem 1.3]{PRS22} with \eqref{eq:intro-modulus}, we obtain in Corollary~\ref{cor:PRS-free} that
\begin{equation}
\label{eq:intro-PRS-free}
\mathcal{R}_{\FF_r,p}(B_N) 
\leq \frac{q^{\frac{pN}{4}}}{1+\frac{q-1}{q+1}\cdot\frac{pN}{2}}.
\end{equation}
Since \eqref{eq:intro-modulus} is exact, \eqref{eq:intro-PRS-free} is obtained by direct substitution into the local transference estimate of \cite{PRS22}. For comparison, Proposition~\ref{prop:finite-support-ratio} gives the elementary estimate
\begin{equation}
\label{eq:intro-finite-support}
\mathcal{R}_{G,p}(\Omega)
\leq |\Omega|^{|\frac{1}{p}-\frac{1}{2}|}
\end{equation}
for any $1 < p < \infty$ and any nonempty finite subset $\Omega$ of a discrete group $G$. On free balls, this yields
\[
\mathcal{R}_{\FF_r,p}(B_N)
\leq 2^{\frac{\alpha_p}{2}}q^{\frac{\alpha_p N}{2}},
\quad \alpha_p \ov{\mathrm{def}}{=} |1-\frac{2}{p}|.
\]
For each even integer $p \geq 2$, we have
\[
\frac{p}{4}-\frac{\alpha_p}{2}
= \frac{p^2-2p+4}{4p}
> 0.
\]
Consequently, for fixed $p$ and $q$, the upper bound in \eqref{eq:intro-PRS-free} is asymptotically larger than the cardinality bound as $N \to \infty$. At $p = 2$, the cardinality estimate gives the exact value $\mathcal{R}_{\FF_r,2}(B_N) = 1$, whereas the right-hand side of \eqref{eq:intro-PRS-free} grows like $\frac{q^{\frac{N}{2}}}{N}$. Thus the loss in this particular application of local transference does not come from estimating the amenability modulus, whose value is known exactly. On the other hand, \cite[Theorem B1]{GPPW26} shows that $\mathcal{R}_{\FF_r,p}(B_N) \to \infty$ as $N \to \infty$ when $|\frac{1}{p}-\frac{1}{2}| > \frac{1}{6}$, without a rate. Our next result gives a polynomial rate for any $p \neq 2$.

\begin{thm}
\label{thm:intro-finite}
Suppose that $1 < p < \infty$ and $p \neq 2$. Put $\alpha_p \ov{\mathrm{def}}{=} |1-\frac{2}{p}|$. There exist constants $c_p > 0$ and $N_p \geq 1$ such that, for any $2 \leq r < \infty$ and any integer $N \geq N_p$,
\begin{equation}
\label{eq:intro-ratio}
\mathcal{R}_{\FF_r,p}(B_N) 
\geq c_p N^{\alpha_p}.
\end{equation}
More precisely, the symbols $a_N(s) \ov{\mathrm{def}}{=} \i^{|s|}1_{B_N}(s)$ satisfy
\[
\norm{M_{a_N}^\HS}_{\cb,S^p_{\FF_r} \to S^p_{\FF_r}}
\lesssim (p^*)^3,
\quad
\norm{M_{a_N}}_{\L^p(\VN(\FF_r)) \to \L^p(\VN(\FF_r))}
\geq c_p N^{\alpha_p}.
\]
There also exist smooth radial cutoffs $b_N$, supported in $B_N$, for which
\begin{equation}
\label{eq:intro-smooth}
\norm{M_{b_N}^\HS}_{\cb,S^p_{\FF_r} \to S^p_{\FF_r}}
\lesssim p^*,
\quad
\norm{M_{b_N}}_{\L^p(\VN(\FF_r)) \to \L^p(\VN(\FF_r))}
\geq c_p N^{\alpha_p}.
\end{equation}
The implicit constants in the Herz--Schur estimates are independent of $p$, $r$ and $N$.
\end{thm}

The Fourier lower estimate is detected on the abelian radial subalgebra of the von Neumann algebra $\VN(\FF_r)$. We use smooth finite sums of its orthogonal polynomials. Without modulation these sums concentrate at an endpoint of the radial spectrum, where the spectral measure has quadratic vanishing in angular coordinates. Multiplication of their coefficients by $\e^{\i k\vartheta}$ moves the concentration to an interior point. We give two-sided norm estimates for both finite sums. The argument is a finite-frequency version of the concentration mechanism used in \cite[Lemma 2.4 and Theorem B2]{GPPW26}.

The same estimates, combined with the Hankel-matrix formula \cite{HSS10} of Haagerup, Steenstrup and Szwarc, give a two-sided result for damped phases.

\begin{thm}
\label{thm:intro-abel}
Suppose that $1 < p < \infty$ and $p \neq 2$. Put $\alpha_p \ov{\mathrm{def}}{=} |1-\frac{2}{p}|$. Fix an integer $r \geq 2$ and an angle $0 < \vartheta < \pi$. For $t > 0$, define $a_{t,\vartheta}(s) \ov{\mathrm{def}}{=} \e^{(-t+\i\vartheta)|s|}$ for $s \in \FF_r$. For $0 < t \leq 1$, we have
\begin{equation}
\label{eq:intro-abel}
\norm{M_{a_{t,\vartheta}}}_{\L^p(\VN(\FF_r)) \to \L^p(\VN(\FF_r))}
\approx_{p,\vartheta}
\norm{M_{a_{t,\vartheta}}}_{\cb,\L^p(\VN(\FF_r)) \to \L^p(\VN(\FF_r))}
\approx_{p,\vartheta} t^{-\alpha_p},
\end{equation}
where the constants may be chosen independently of $2 \leq r < \infty$. On the other hand,
\[
\sup_{t>0}\norm{M_{a_{t,\vartheta}}^\HS}_{\cb,S^p_{\FF_r} \to S^p_{\FF_r}} 
\lesssim p^*.
\]
\end{thm}
The same damped-phase estimates also give the exact order $N^{\alpha_p}$ of the ordinary and completely bounded Fourier multiplier norms of the smooth radial cutoffs appearing in Theorem~\ref{thm:intro-finite}, see Corollary~\ref{cor:smooth-two-sided}.

Combining Theorem~\ref{thm:intro-finite}, \eqref{eq:intro-finite-support} and \eqref{eq:intro-modulus}, we obtain in Corollary~\ref{cor:free-ratio-bounds} that, for $1 < p < \infty$ with $p \neq 2$ and any sufficiently large integer $N$,
\begin{equation}
\label{eq:intro-two-sided}
c_p\left(\frac{2}{\log q}\log\frac{1}{1-\delta_{\FF_r}(B_N)}\right)^{\alpha_p}
\leq \mathcal R_{\FF_r,p}(B_N)
\leq 2^{\frac{\alpha_p}{2}}(1+N)^{\alpha_p}
\big(1-\delta_{\FF_r}(B_N)\big)^{-\alpha_p}.
\end{equation}
In other words, if $L_N \ov{\mathrm{def}}{=} \log\frac{1}{1-\delta_{\FF_r}(B_N)}$, then $\mathcal{R}_{\FF_r,p}(B_N)^{\frac{1}{\alpha_p}}$ lies between a constant multiple of $L_N$ and a constant multiple of $L_N\e^{L_N}$ for sufficiently large $N$, with constants depending only on $p$ and $q$. Indeed, \eqref{eq:intro-modulus} shows that $L_N$ is comparable to $N$, with constants depending only on $q$. Corollary~\ref{cor:smooth-two-sided} shows that the norm ratios of the smooth cutoffs $b_N$ of Theorem~\ref{thm:intro-finite} are exactly of order $N^{\alpha_p}$, so that the lower bound of Theorem~\ref{thm:intro-finite} cannot be improved with these symbols. This is our contribution to \cite[Problem 4.2]{GPPW26} on free balls.

For radial symbols, the exponential upper bound can be replaced by a polynomial one. Let $\mathcal{R}^{\mathrm{rad}}_{\FF_r,p}(B_N)$ denote the supremum in \eqref{eq:local-ratio} restricted to radial symbols, that is, to functions of the word length. Using the Hankel-matrix formula of \cite{HSS10}, we show in Proposition~\ref{prop:radial-polynomial} that
\begin{equation}
\label{eq:intro-radial-polynomial}
c_pN^{\alpha_p}
\leq \mathcal{R}^{\mathrm{rad}}_{\FF_r,p}(B_N)
\leq 2^{\frac{\alpha_p}{2}}(N+2)^{\frac{3\alpha_p}{2}}
\end{equation}
for $p \neq 2$ and any sufficiently large integer $N$, uniformly in $r$. Thus the radial Fourier--Schur ratio of free balls grows at least like $N^{\alpha_p}$ and at most like $N^{\frac{3\alpha_p}{2}}$. Since $L_N$ is comparable to $N$, the same holds with $N$ replaced by $L_N$.

\paragraph{Structure of the paper}
Section~\ref{sec:preliminaries} recalls the required norm conventions and the radial phase estimate. Section~\ref{sec:automatic} proves Theorem~\ref{thm:intro-automatic} and its geometric extensions.  Section~\ref{sec:constants} determines the optimal dependence on $p$ for positive definite even symbols. Section~\ref{sec:modulus} computes the local amenability modulus of free balls, proves the elementary estimate \eqref{eq:intro-finite-support} and deduces Corollary~\ref{cor:PRS-free}. Section~\ref{sec:packets} uses the spectral model of the radial subalgebra of $\VN(\FF_r)$ to prove Theorem~\ref{thm:intro-finite} and establish the quantitative bounds of Corollary~\ref{cor:free-ratio-bounds}. It then treats damped radial phases, proves Theorem~\ref{thm:intro-abel} and Corollary~\ref{cor:smooth-two-sided}, and establishes the polynomial bound \eqref{eq:intro-radial-polynomial} for radial symbols.

\section{Preliminaries}
\label{sec:preliminaries}

We use $\N \ov{\mathrm{def}}{=} \{1,2,\ldots\}$, $\N_0 \ov{\mathrm{def}}{=} \{0,1,2,\ldots\}$ and equip $\T$ with its normalized Haar measure $\mu_\T$. All Hilbert spaces are complex, and their inner products are linear in the first variable. For $1 < p < \infty$, we write $p' \ov{\mathrm{def}}{=} \frac{p}{p-1}$ and $p^* \ov{\mathrm{def}}{=} \max\{p,p'\}$. The notation $A \lesssim B$ means that $A \leq C B$ for a numerical constant $C$. Subscripts indicate the permitted dependence of a constant. We use $A \approx B$ when both inequalities hold.

For a countable set $I$, we abbreviate $S^p(\ell^2_I)$ to $S^p_I$. We use the unnormalized trace $\Tr$ on matrix algebras and on $\B(\ell^2_I)$. Thus $\norm{x}_{S^p_I} \ov{\mathrm{def}}{=} (\Tr|x|^p)^{\frac{1}{p}}$. Rectangular Schatten classes are defined by the singular values, or equivalently by embedding a rectangular matrix into a larger square matrix with zero entries elsewhere. If $T \co S^p_I \to S^p_I$ acts on a Schatten class, its completely bounded norm is
\begin{equation}
\label{eq:cb-definition}
\norm{T}_{\cb,S^p_I \to S^p_I} 
= \sup_{n \geq 1}\bnorm{\Id_{S^p_n} \ot T}_{S^p_n(S^p_I) \to S^p_n(S^p_I)}.
\end{equation}
This agrees with the completely bounded norm for the canonical operator space structure, see \cite{Pis98}. We use $\norm{T}_{S^p_I \to S^p_I}$ for the ordinary norm. Norms of algebraically defined multipliers are assigned the value $\infty$ if no bounded extension exists. We record the following standard result. See \cite[Lemma 4.3]{ArK22} for the last part.

\begin{lemma}
\label{lem:finite-corners}
Let $1 \leq p < \infty$ and let $a \co I \times I \to \Cbb$ be a complex function.
\begin{enumerate}
\item The ordinary norm of its Schur multiplier is the supremum of the norms of its finite square corners. The analogous assertion holds at any fixed matrix level and for the completely bounded norm.
\item A finite rectangular restriction has norm at most that of the original multiplier, at every matrix level.
\item Suppose that $a_n(x,y) \to a(x,y)$ for any $x,y \in I$. If $\sup_n \norm{M_{a_n}}_{S^p_I \to S^p_I} \leq C$, then $\norm{M_a}_{S^p_I \to S^p_I} \leq C$. The same statement holds for completely bounded norms.
\end{enumerate}
\end{lemma}

For a discrete group $G$, we denote by $\lambda$ its left regular representation and equip $\VN(G)$ with its canonical trace $\tau(x) \ov{\mathrm{def}}{=} \la x\delta_e,\delta_e\ra$. We use the corresponding tracial spaces $\L^p(\VN(G))$. The completely bounded norm of a Fourier multiplier is obtained by taking the supremum of its extensions to $\L^p(\M_k \otvn \VN(G),\Tr\ot\tau)$. In particular, on the abelian group $\Z$ this is the norm of the extension to the Bochner space $\L^p(\T,S^p)$.

For $b \co \Z \to \mathbb{C}$, let $M_b(z^n) \ov{\mathrm{def}}{=} b(n)z^n$ and let $M_b^\HS$ be the Toeplitz Schur multiplier with symbol $b(i-j)$ on $S^p_\Z$. We shall use
\begin{equation}
\label{eq:Z-transference}
\norm{M_b}_{\L^p(\T) \to \L^p(\T)} 
\leq \norm{M_b^\HS}_{S^p_\Z \to S^p_\Z} 
\leq \norm{M_b^\HS}_{\cb,S^p_\Z \to S^p_\Z} 
= \norm{M_b}_{\cb,\L^p(\T) \to \L^p(\T)}.
\end{equation}
The first inequality is \cite[Theorem 2.1 p.~1167]{NeR11}, and the equality is \cite[Theorem 2.7 p.~1171]{NeR11}. The same Toeplitz norm is obtained by restricting both indices to $\N_0$. Indeed, any finite subset of $\Z$ can be translated into $\N_0$, without changing its differences.

Let $P_+ \co \L^p(\T) \to \L^p(\T)$ denote the Riesz projection with symbol $1_{\N_0}$. We put $
r_{\cb,p} \ov{\mathrm{def}}{=}
\norm{P_+}_{\cb,\L^p(\T) \to \L^p(\T)}$. According to \cite{HoV00}, we have
\begin{equation}
\label{norm-P+}
\norm{P_+}_{\L^p(\T) \to \L^p(\T)}
= \frac{1}{\sin(\frac{\pi}{p})}.
\end{equation}
For $j \in \Z$, let $U_jf(z) = z^jf(z)$, and let $Rf(z) = f(\overline z)$. These maps are complete isometries on $\L^p(\T)$, and
\[
M_{b(\cdot-j)} = U_jM_bU_j^{-1},
\quad
M_{b(-\cdot)} = RM_bR.
\]
Thus translation and reflection of a Fourier symbol preserve both its ordinary and completely bounded multiplier norms. We refer to \cite{Lar71} for more information on Fourier multipliers on locally compact abelian groups. 

We next recall the part of the tree multiplier theorem that will be used throughout. If $z \in \T$, we introduce
\begin{equation}
\label{eq:kappa-definition}
\Delta_z^X([a_{x,y}]) \ov{\mathrm{def}}{=} \big[z^{\dist_X(x,y)}a_{x,y}\big], \quad 
\kappa_p \ov{\mathrm{def}}{ = } \sup_X \sup_{z \in \T} \norm{\Delta_z^X}_{\cb, S^p_X \to S^p_X}
\end{equation}
where the first supremum runs over countable trees. Since $\Delta_1^X = \Id$, we have $\kappa_p \geq 1$. The bound $\kappa_p \lesssim p^*$ is the inequality labelled (Sph$_{\mathrm{T}}$) in \cite{GPPW26}. It is stated there for scalar matrices, but its proof relies on noncommutative martingale inequalities for the filtration of $\B(\ell^2_X)$ by block-diagonal conditional expectations, and these inequalities hold with the same constant for the amplified filtration of $\B(\ell^2_X) \ot \M_n$. Hence the same bound holds at every matrix level.

\begin{thm}
\label{thm-tree-input}
For $1 < p < \infty$, we have $1 \leq \kappa_p \lesssim p^*$. Moreover, for any countable tree $X$ and any complex function $m \co \N_0 \to \mathbb{C}$, we have
\begin{equation}
\label{eq:tree-upper}
\norm{M_{m \circ \dist_X}}_{\cb,S^p_X \to S^p_X} 
\leq \kappa_p^2\norm{M_{\widetilde m}}_{\cb,\L^p(\T) \to \L^p(\T)}.
\end{equation}
If $X$ has infinite diameter, then we have
\begin{equation}
\label{eq:tree-lower}
\norm{M_{\widetilde m}}_{\cb,\L^p(\T) \to \L^p(\T)} 
\leq \norm{M_{m \circ \dist_X}}_{\cb,S^p_X \to S^p_X},
\quad
\norm{M_{\widetilde m}}_{\L^p(\T) \to \L^p(\T)} 
\leq \norm{M_{m \circ \dist_X}}_{S^p_X \to S^p_X}.
\end{equation}
\end{thm}

\begin{proof}
Let $a \in S^p_X$ be a matrix with finite support. The map $J(a) \co \T \to S^p_X$, $z \mapsto \Delta_z^X(a)$, is a trigonometric polynomial whose coefficient of $z^n$ is the restriction of $a$ to the entries $(x,y)$ with $\dist_X(x,y) = n$. Hence
\[
J\big(M_{m \circ \dist_X}(a)\big) 
= (M_{\widetilde m} \ot \Id_{S^p_X})\big(J(a)\big).
\]
Since $\Delta_{\overline z}^X\Delta_z^X = \Id$ for any $z \in \T$, we have $\norm{a}_{S^p_X} \leq \kappa_p\norm{J(a)(z)}_{S^p_X}$ and $\norm{J(a)(z)}_{S^p_X} \leq \kappa_p\norm{a}_{S^p_X}$. Integrating over $\T$, we obtain
\[
\norm{M_{m \circ \dist_X}(a)}_{S^p_X}
\leq \kappa_p\bnorm{(M_{\widetilde m} \ot \Id_{S^p_X})\big(J(a)\big)}_{\L^p(\T,S^p_X)}
\leq \kappa_p^2\norm{M_{\widetilde m}}_{\cb,\L^p(\T) \to \L^p(\T)}\norm{a}_{S^p_X}.
\]
The same argument applies at every matrix level, and Lemma~\ref{lem:finite-corners} gives \eqref{eq:tree-upper}. If $X$ has infinite diameter, it contains geodesic segments of any finite length, and the compression of $M_{m \circ \dist_X}$ to such a segment is a finite corner of the Toeplitz Schur multiplier $M_{\widetilde m}^\HS$. Hence \eqref{eq:tree-lower} follows from \eqref{eq:Z-transference} and Lemma~\ref{lem:finite-corners}.
\end{proof}

We refer to the books \cite{Die25} and \cite{Ser80} for more information on trees.

\section{Geometric multiplicity and automatic complete boundedness}
\label{sec:automatic}

We isolate the finite configuration used to recover matrix amplifications. This makes it possible to distinguish infinite degree from local finiteness.

\begin{defi}
\label{def:multiplicity}
Let $h \geq 1$ be an integer. We say that a tree $X$ satisfies the rectangular multiplicity condition at depth $h$ if for any integers $L,k \geq 1$ there exist pairwise distinct vertices
\[
x_{i,a},\ y_{j,b},\quad 0 \leq i,j \leq L,\quad 1 \leq a,b \leq k,
\]
such that
\begin{equation}
\label{eq:multiplicity-distance}
\dist_X(x_{i,a},y_{j,b}) = |i-j|+2h.
\end{equation}
\end{defi}

Only the distances between the two families are prescribed. The configurations for different $L$ and $k$ need not be compatible. The condition implies that $X$ has infinite diameter.

\begin{example} \normalfont
\label{X-infty-example}
Fix integers $L,k \geq 1$ and choose a bi-infinite geodesic $(v_n)_{n \in \Z}$ in $X_\infty$. For each $i \in \{0,\ldots,L\}$, choose $2k$ distinct neighbours of $v_i$ outside this geodesic, and label them $x_{i,1},\ldots,x_{i,k},y_{i,1},\ldots,y_{i,k}$. This is possible because $v_i$ has infinitely many neighbours, only two of which lie on the geodesic. All the chosen vertices are pairwise distinct: a vertex outside the geodesic cannot be adjacent to two different vertices of the geodesic, since this would create a cycle.

For $0 \leq i,j \leq L$ and $1 \leq a,b \leq k$, the unique path from $x_{i,a}$ to $y_{j,b}$ consists of the edge from $x_{i,a}$ to $v_i$, the geodesic segment from $v_i$ to $v_j$, and the edge from $v_j$ to $y_{j,b}$. Therefore, $\dist_{X_\infty}(x_{i,a},y_{j,b}) = 1 + |i-j| + 1 = |i-j| + 2$. When $i = j$, the middle segment has length zero, while $x_{i,a} \neq y_{i,b}$, so the distance is still $2$. Thus the tree $X_\infty$ satisfies Definition~\ref{def:multiplicity} with $h = 1$. See Figure~\ref{fig:rectangular-multiplicity}.
\end{example}

\begin{figure}[htbp]
\centering
\begin{tikzpicture}[
  x = 1cm, y = 1cm,
  vertex/.style = {circle, fill = black, inner sep = 1.65pt},
  edge/.style = {draw = black!45, line width = 0.45pt},
  long path/.style = {draw = black, line width = 1.1pt},
  short path/.style = {draw = black, line width = 1.1pt,
                       dash pattern = on 3pt off 2pt},
  every node/.style = {font = \small}
]
  \foreach \idx in {0,1,2} {
    \coordinate (v\idx) at ({4.1*\idx},0);
    \foreach \arm/\offset in {1/-1.05,2/0,3/1.05} {
      \coordinate (x\idx-\arm) at ({4.1*\idx+\offset},1.55);
      \coordinate (y\idx-\arm) at ({4.1*\idx+\offset},-1.55);
      \draw[edge] (x\idx-\arm) -- (v\idx) -- (y\idx-\arm);
    }
  }
  \draw[edge] (v0) -- (v1) -- (v2);
  \draw[edge,dashed] (-1.8,0) -- (v0);
  \draw[edge,dashed] (v2) -- (10,0);
  \node[anchor = east] at (-1.9,0) {$\cdots$};
  \node[anchor = west] at (10.1,0) {$\cdots$};

  \draw[long path] (x0-2) -- (v0) -- (v1) -- (v2) -- (y2-2);
  \draw[short path] (x1-2) -- (v1) -- (y1-2);

  \foreach \idx in {0,1,2} {
    \node[vertex] at (v\idx) {};
    \node[anchor = south west] at ({4.1*\idx+0.16},0.10)
      {$v_{\idx}$};
    \foreach \arm in {1,2,3} {
      \node[vertex] at (x\idx-\arm) {};
      \node[anchor = south] at ({4.1*\idx+(\arm-2)*1.05},1.68)
        {$x_{\idx,\arm}$};
      \node[vertex] at (y\idx-\arm) {};
      \node[anchor = north] at ({4.1*\idx+(\arm-2)*1.05},-1.68)
        {$y_{\idx,\arm}$};
    }
  }
  \node[anchor = north] at (2.05,-0.12) {$1$};
  \node[anchor = north] at (6.15,-0.12) {$1$};

  \draw[long path] (-0.85,-2.52) -- (-0.18,-2.52);
  \node[anchor = west] at (-0.05,-2.52)
    {$\dist_{X_\infty}(x_{0,2},y_{2,2}) = 1 + 2 + 1 = 4$};
  \draw[short path] (5.90,-2.52) -- (6.57,-2.52);
  \node[anchor = west] at (6.70,-2.52)
    {$\dist_{X_\infty}(x_{1,2},y_{1,2}) = 2$};
\end{tikzpicture}
\caption{The finite configuration for $L = 2$ and $k = 3$. Every edge has graph length $1$; unselected vertices and branches of $X_\infty$ are omitted. The solid thick path joins $x_{0,2}$ to $y_{2,2}$ and has length $4$. The dashed thick path joins $x_{1,2}$ to $y_{1,2}$ through $v_1$ and has length $2$.}
\label{fig:rectangular-multiplicity}
\end{figure}
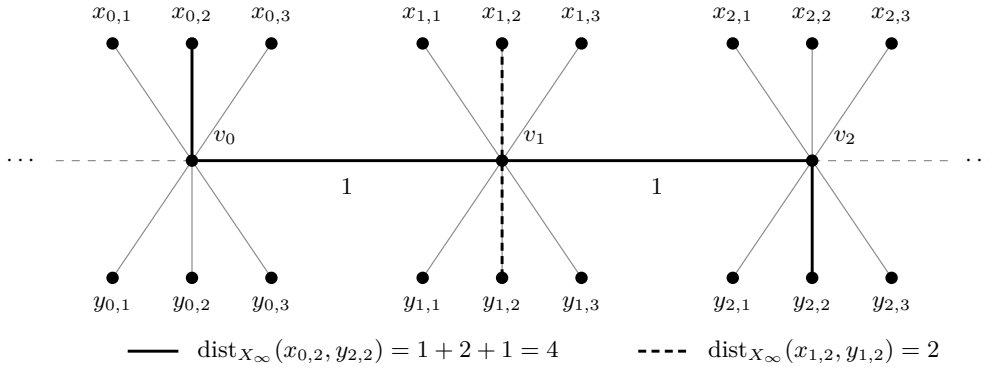

\begin{example} \normalfont
\label{ex:locally-finite}
There exists a locally finite tree $X$ satisfying Definition~\ref{def:multiplicity} with $h = 1$. Start with a ray $v_0,v_1,\ldots$. For $n \geq 1$, put $a_n \ov{\mathrm{def}}{=}\frac{n(n-1)}{2}$ and $I_n \ov{\mathrm{def}}{=} n(n-1) \{v_{a_n},\ldots,v_{a_n + n - 1}\}$. These sets partition the ray into consecutive blocks of $n$ vertices. Attach $2n$ new leaves to each vertex in $I_n$. Each attached leaf has degree $1$, and each vertex in $I_n$ has degree at most $2n + 2$. Thus $X$ is locally finite.

Given integers $L,k \geq 1$, choose $n \geq \max\{L + 1,k\}$ and put $w_i \ov{\mathrm{def}}{=} v_{a_n + i}$ for $0 \leq i \leq L$. At each $w_i$, choose $2k$ distinct attached leaves and label them $x_{i,1},\ldots,x_{i,k},y_{i,1},\ldots,y_{i,k}$. All these leaves are pairwise distinct. The unique path from $x_{i,a}$ to $y_{j,b}$ consists of the edge from $x_{i,a}$ to $w_i$, the segment of the ray from $w_i$ to $w_j$, and the edge from $w_j$ to $y_{j,b}$. Consequently, $\dist_X(x_{i,a},y_{j,b}) = 1 + |i-j| + 1 = |i-j| + 2$, including when $i = j$. This proves the required multiplicity condition; see Figure~\ref{fig:locally-finite-multiplicity}.
\end{example}

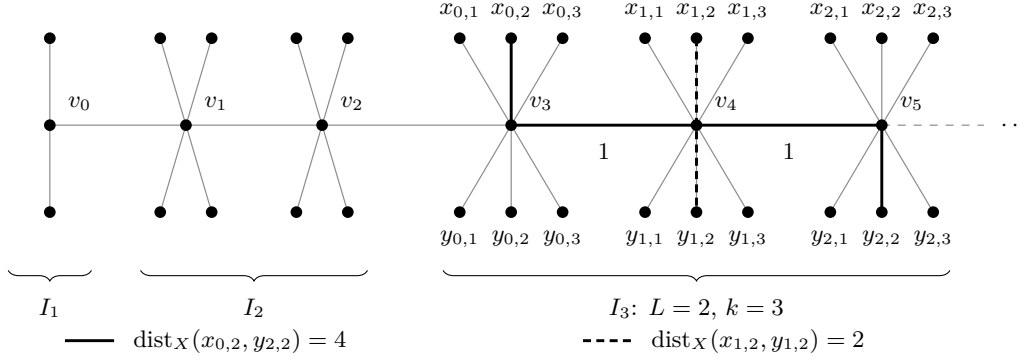
\begin{figure}[htbp]
\centering
\begin{tikzpicture}[
  x = 1cm, y = 1cm,
  vertex/.style = {circle, fill = black, inner sep = 1.55pt},
  edge/.style = {draw = black!45, line width = 0.45pt},
  long path/.style = {draw = black, line width = 1.1pt},
  short path/.style = {draw = black, line width = 1.1pt,
                       dash pattern = on 3pt off 2pt},
  block/.style = {decorate,
                 decoration = {brace, mirror, amplitude = 4pt}},
  every node/.style = {font = \small}
]
  \foreach \idx/\pos/\size in {
    0/0/1,1/1.8/2,2/3.6/2,3/6.1/3,4/8.55/3,5/11/3
  } {
    \coordinate (v\idx) at (\pos,0);
    \foreach \arm in {1,...,\size} {
      \pgfmathsetmacro{\offset}{0.68*(\arm-(\size+1)/2)}
      \coordinate (u\idx-\arm) at ({\pos+\offset},1.15);
      \coordinate (z\idx-\arm) at ({\pos+\offset},-1.15);
      \draw[edge] (u\idx-\arm) -- (v\idx) -- (z\idx-\arm);
    }
  }

  \draw[edge] (v0) -- (v1) -- (v2) -- (v3) -- (v4) -- (v5);
  \draw[edge,dashed] (v5) -- (12.35,0);
  \node[anchor = west] at (12.45,0) {$\cdots$};

  \draw[long path] (u3-2) -- (v3) -- (v4) -- (v5) -- (z5-2);
  \draw[short path] (u4-2) -- (v4) -- (z4-2);

  \foreach \idx/\pos/\size in {
    0/0/1,1/1.8/2,2/3.6/2,3/6.1/3,4/8.55/3,5/11/3
  } {
    \node[vertex] at (v\idx) {};
    \node[anchor = south west] at ({\pos+0.12},0.08)
      {$v_{\idx}$};
    \foreach \arm in {1,...,\size} {
      \node[vertex] at (u\idx-\arm) {};
      \node[vertex] at (z\idx-\arm) {};
    }
  }

  \foreach \idx/\pos in {0/6.1,1/8.55,2/11} {
    \foreach \arm in {1,2,3} {
      \node[anchor = south] at ({\pos+0.68*(\arm-2)},1.27)
        {$x_{\idx,\arm}$};
      \node[anchor = north] at ({\pos+0.68*(\arm-2)},-1.27)
        {$y_{\idx,\arm}$};
    }
  }

  \node[anchor = north] at (7.325,-0.12) {$1$};
  \node[anchor = north] at (9.775,-0.12) {$1$};

  \draw[block] (-0.55,-1.92) -- (0.55,-1.92)
    node[midway,below = 7pt] {$I_1$};
  \draw[block] (1.20,-1.92) -- (4.20,-1.92)
    node[midway,below = 7pt] {$I_2$};
  \draw[block] (5.20,-1.92) -- (11.90,-1.92)
    node[midway,below = 7pt] {$I_3$: $L = 2$, $k = 3$};

  \draw[long path] (0.20,-2.85) -- (0.85,-2.85);
  \node[anchor = west] at (0.98,-2.85)
    {$\dist_X(x_{0,2},y_{2,2}) = 4$};
  \draw[short path] (7.05,-2.85) -- (7.70,-2.85);
  \node[anchor = west] at (7.83,-2.85)
    {$\dist_X(x_{1,2},y_{1,2}) = 2$};
\end{tikzpicture}
\caption{The first three blocks of the locally finite tree $X$. Each vertex in $I_n$ has $2n$ attached leaves. Within $I_3$, we choose $L = 2$ and $k = 3$, with $x_{i,a}$ and $y_{i,b}$ attached to $v_{i+3}$. Here, every edge has graph length $1$. The solid thick path has length $4$, while the dashed thick path through $v_4$ has length $2$. The remaining blocks are omitted.}
\label{fig:locally-finite-multiplicity}
\end{figure}

\begin{lemma}
\label{lem:shifted-detection}
Suppose that $1 < p < \infty$. Fix an integer $h \geq 1$. Assume that the tree $X$ satisfies the rectangular multiplicity condition at depth $h$. Let $m \co \N_0 \to \mathbb{C}$ be a complex function and suppose that the matrix $m \circ \dist_X$ induces a bounded Schur multiplier $M_{m \circ \dist_X} \co S^p_X \to S^p_X$. Define the complex function $b_h \co \Z \to \mathbb{C}$ by
\begin{equation}
\label{eq:shifted-symbol}
b_h(n) 
\ov{\mathrm{def}}{=} m(|n|+2h),\quad n \in \Z.
\end{equation}
Then the function $b_h$ induces a completely bounded Fourier multiplier $M_{b_h} \co \L^p(\T) \to \L^p(\T)$ and
\begin{equation}
\label{eq:shifted-bound}
\norm{M_{b_h}}_{\cb,\L^p(\T) \to \L^p(\T)} 
\leq \norm{M_{m \circ \dist_X}}_{S^p_X \to S^p_X}.
\end{equation}
\end{lemma}

\begin{proof}
Put $
C \ov{\mathrm{def}}{=}
\norm{M_{m \circ \dist_X}}_{S^p_X \to S^p_X}$. Fix integers $L,k \geq 1$, and choose pairwise distinct vertices $x_{i,a}$ and $y_{j,b}$ satisfying \eqref{eq:multiplicity-distance}. Consider the complex Hilbert space $
H_{L,k}
\ov{\mathrm{def}}{=}
\ell^2_{\{0,\ldots,L\}} \ot \mathbb{C}^k$, with its standard orthonormal basis $(e_i \ot e_a)_{0 \leq i \leq L, 1 \leq a \leq k}$. A block matrix $A = [A_{ij}]_{i,j = 0}^L$, with $A_{ij} \in \M_k$, acts on the space $H_{L,k}$ by
\[
A(e_j \ot e_b)
=
\sum_{i = 0}^L \sum_{a = 1}^k [A_{ij}]_{ab} e_i \ot e_a.
\]
Write $(e_z)_{z \in X}$ for the standard orthonormal basis of the Hilbert space $\ell^2_X$. Define two isometries $U,V \co H_{L,k} \to \ell^2_X$ by
\[
U(e_i \ot e_a) 
\ov{\mathrm{def}}{=} e_{x_{i,a}}
\quad \text{and} \quad
V(e_j \ot e_b) 
\ov{\mathrm{def}}{=} e_{y_{j,b}}.
\]
They are isometries because the selected vertices are pairwise distinct. Their ranges are the subspaces
\[
\mathcal H_x
= \Span\{e_{x_{i,a}}\}
\quad \text{and} \quad
\mathcal H_y
= \Span \{e_{y_{j,b}}\}.
\]
Associate to $A$ the finite-rank operator $J(A) \ov{\mathrm{def}}{=} UAV^*$ on the Hilbert space $\ell^2_X$. This operator maps $\mathcal H_y$ into $\mathcal H_x$ and vanishes on $\mathcal H_y^\perp$. More explicitly, we have
\[
J(A)e_{y_{j,b}}
= \sum_{i = 0}^L \sum_{a = 1}^k (A_{ij})_{ab} e_{x_{i,a}}.
\]
Thus its matrix coefficient at $(x_{i,a},y_{j,b})$ is $[A_{ij}]_{ab}$, and all coefficients outside this rectangular corner are zero. Since $U^*U = \Id_{H_{L,k}}$, we have
\[
J(A)^*J(A) 
= VA^*AV^*.
\]
The operator on the right is unitarily equivalent to $A^*A$ on $\mathcal H_y$ and is zero on
$\mathcal H_y^\perp$. Consequently, $J(A)$ and $A$ have the same nonzero singular values, with the same multiplicities. Under the canonical identification $S^p(H_{L,k}) = S^p_{L+1}(S^p_k)$, this gives
\[
\norm{J(A)}_{S^p_X}
=
\bnorm{[A_{ij}]_{i,j = 0}^L}_{S^p_{L+1}(S^p_k)}.
\]
Now, we compute the action of the Schur multiplier. For any $i,j,a,b$, its defining formula gives
\begin{align*}
\MoveEqLeft
\bigl(M_{m \circ \dist_X}(J(A))\bigr)_{x_{i,a},y_{j,b}}
=
m\bigl(\dist_X(x_{i,a},y_{j,b})\bigr)[A_{ij}]_{ab}
\\
&\ov{\eqref{eq:multiplicity-distance}}{=}
m(|i-j|+2h)[A_{ij}]_{ab}
\ov{\eqref{eq:shifted-symbol}}{=}
b_h(i-j)[A_{ij}]_{ab}.
\end{align*}
All coefficients outside the selected rectangular
corner remain zero. Hence the preceding coefficient
identities prove the operator identity
\[
M_{m \circ \dist_X}(J(A))
= J\bigl([b_h(i-j)A_{ij}]_{i,j = 0}^L\bigr).
\]
Applying the norm identity for $J$ and the boundedness of $M_{m \circ \dist_X}$, we obtain
\begin{align}
\label{eq:amplified-corner}
\MoveEqLeft
\bnorm{[b_h(i-j)A_{ij}]_{i,j = 0}^L}_{S^p_{L+1}(S^p_k)}
=
\norm{M_{m \circ \dist_X}(J(A))}_{S^p_X}
\\
&\leq C\norm{J(A)}_{S^p_X}
= C\bnorm{[A_{ij}]_{i,j = 0}^L}_{S^p_{L+1}(S^p_k)}. \nonumber
\end{align}
The map in \eqref{eq:amplified-corner} is the $k$-th matrix amplification of the Toeplitz Schur multiplier on $\{0,\ldots,L\}$, under the canonical identification of the tensor factors. Its norm is therefore at most $C$, independently of $L$ and $k$.

Any finite subset of $\Z$ can be translated into $\{0,\ldots,L\}$ for some $L$. Such a translation preserves all differences $i-j$, and extending a matrix by zero preserves its Schatten norm. Thus \eqref{eq:amplified-corner} bounds every finite corner of the Toeplitz Schur multiplier at every
matrix level. Lemma~\ref{lem:finite-corners} gives
\[
\norm{M_{b_h}^\HS}_{\cb,S^p_\Z \to S^p_\Z} 
\leq C.
\]
Finally, \eqref{eq:Z-transference} yields a completely bounded Fourier multiplier $M_{b_h}$ satisfying
\[
\norm{M_{b_h}}_{\cb,\L^p(\T) \to \L^p(\T)}
\ov{\eqref{eq:Z-transference}}{=}
\norm{M_{b_h}^\HS}_{\cb,S^p_\Z \to S^p_\Z}
\leq C.
\]
\end{proof}

The shift $2h$ can be removed as the following result shows.

\begin{prop}
\label{prop:automatic-general}
Suppose that $1 < p < \infty$. Fix an integer $h \geq 1$. Assume that the tree $X$ satisfies the rectangular multiplicity condition at depth $h$. Let $m \co \N_0 \to \mathbb{C}$ be a complex function. If the function $m \circ \dist_X$ induces a bounded Schur multiplier $M_{m \circ \dist_X} \co S^p_X \to S^p_X$ then
\begin{equation}
\label{eq:general-automatic}
\norm{M_{\widetilde m}}_{\cb,\L^p(\T) \to \L^p(\T)} 
\leq (2r_{\cb,p}+\sqrt{4h-1})\norm{M_{m \circ \dist_X}}_{S^p_X \to S^p_X}.
\end{equation}
Moreover, the Schur multiplier $M_{m \circ \dist_X} \co S^p_X \to S^p_X$ is completely bounded and
\begin{equation}
\label{eq:general-automatic-schur}
\norm{M_{m \circ \dist_X}}_{\cb,S^p_X \to S^p_X} 
\leq \kappa_p^2(2r_{\cb,p}+\sqrt{4h-1}) \norm{M_{m \circ \dist_X}}_{S^p_X \to S^p_X}.
\end{equation}
\end{prop}

\begin{proof}
Put $C \ov{\mathrm{def}}{=} \norm{M_{m \circ \dist_X}}_{S^p_X \to S^p_X}$. Since $X$ is a connected tree of infinite diameter, every nonnegative integer occurs as a distance. Testing the Schur multiplier on the matrix units $e_{xy}$, which have norm one in $S^p_X$, we obtain
\begin{equation}
\label{inter-34}
\norm{m}_{\ell^\infty_{\mathbb{N}_0}}
= \sup_{x,y \in X}\big|m(\dist_X(x,y))\big|
\leq C.
\end{equation}
With the function $b_h$ as in \eqref{eq:shifted-symbol}, we have the identity
\begin{align}
\label{eq:unshift}
\widetilde{m}(n)
=m(|n|)
=1_{\{n \leq -2h\}}b_h(n+2h)
+m(|n|)1_{\{|n|<2h\}}+1_{\{n \geq 2h\}} b_h(n-2h),\quad n \in \Z.
\end{align}
Let $c_h(n) \ov{\mathrm{def}}{=} m(|n|)1_{\{|n|<2h\}}$ and put $K_h(z) \ov{\mathrm{def}}{=} \sum_{|j|<2h}m(|j|)z^{-j}$. Since we have $c_h(n) = \int_\T z^nK_h(z)\d\mu_\T(z)$ for any integer $n$, integration against $K_h\mu_\T$ gives
\[
\norm{M_{c_h}}_{\cb,\L^p(\T) \to \L^p(\T)}
\leq \norm{K_h}_{\L^1(\T)}
\leq \norm{K_h}_{\L^2(\T)}
= \bigg(\sum_{|j|<2h}|m(|j|)|^2\bigg)^{\frac{1}{2}}
\leq \sqrt{4h-1}C.
\]
The first and the last terms induce Fourier multipliers having completely bounded norms at most $r_{\cb,p}\norm{M_{b_h}}_{\cb,\L^p(\T) \to \L^p(\T)}$, by the translation and reflection properties recalled in Section \ref{sec:preliminaries}. Hence
\begin{align}
\MoveEqLeft
\label{inter-466}
\norm{M_{\widetilde m}}_{\cb,\L^p(\T) \to \L^p(\T)} 
\ov{\eqref{eq:unshift}}{\leq} \bnorm{M_{n \mapsto 1_{\{n \geq 2h\}}b_h(n-2h)}}_{\cb,\L^p(\T) \to \L^p(\T)}+\bnorm{M_{c_h}}_{\cb,\L^p(\T) \to \L^p(\T)} \\
&+\bnorm{M_{n \mapsto 1_{\{n \leq -2h\}}b_h(n+2h)}}_{\cb,\L^p(\T) \to \L^p(\T)} \nonumber\\
&\ov{\eqref{inter-34}}{\leq} 2r_{\cb,p}\norm{M_{b_h}}_{\cb,\L^p(\T) \to \L^p(\T)}+\sqrt{4h-1}C
\ov{\eqref{eq:shifted-bound}}{ \leq }(2r_{\cb,p}+\sqrt{4h-1})C. \nonumber        
\end{align}
This proves \eqref{eq:general-automatic}. Finally, using Theorem \ref{thm-tree-input} we obtain
\begin{align*}
\MoveEqLeft
\norm{M_{m \circ \dist_X}}_{\cb,S^p_X \to S^p_X}          
\ov{\eqref{eq:tree-upper}}{\leq} \kappa_p^2\norm{M_{\widetilde m}}_{\cb,\L^p(\T) \to \L^p(\T)} \\
&\ov{\eqref{inter-466}}{\leq} \kappa_p^2(2r_{\cb,p}+\sqrt{4h-1}) \norm{M_{m \circ \dist_X}}_{S^p_X \to S^p_X}.
\end{align*}
This proves \eqref{eq:general-automatic-schur}.
\end{proof}

Recall that $X_\infty$ denotes the countable homogeneous tree in which each vertex has countably infinitely many neighbours.

\begin{cor}
\label{cor-automatic-bis}
Let $1 < p < \infty$ and let $m \co \N_0 \to \Cbb$ be a complex function. The following assertions are equivalent:
\begin{enumerate}
\item the Schur multiplier $M_{m \circ \dist_{X_\infty}}$ is bounded on the Schatten space $S^p_{X_\infty}$,
\item the Schur multiplier $M_{m \circ \dist_{X_\infty}}$ is completely bounded on this space,
\item the Fourier multiplier $M_{\widetilde m}$ is completely bounded on $\L^p(\T)$.
\end{enumerate}
Moreover, if the first assertion holds, then
\begin{equation}
\label{eq:automatic-infinity-estimate}
\norm{M_{\widetilde m}}_{\cb,\L^p(\T) \to \L^p(\T)} 
\leq (2r_{\cb,p}+\sqrt{3}) \norm{M_{m \circ \dist_{X_\infty}}}_{S^p_{X_\infty} \to S^p_{X_\infty}}.
\end{equation}
\end{cor}

\begin{proof}
By Example \ref{X-infty-example}, the tree $X_\infty$ satisfies Definition~\ref{def:multiplicity} with $h = 1$. Proposition~\ref{prop:automatic-general} proves the implication from ordinary boundedness to complete boundedness of the Fourier multiplier $M_{\widetilde m}$, with \eqref{eq:automatic-infinity-estimate}. Theorem~\ref{thm-tree-input} gives the remaining implications. Finally, \eqref{eq:general-automatic-schur} with $h = 1$, together with $\kappa_p \lesssim p^*$ and $r_{\cb,p} \lesssim p^*$, see the proof of Proposition~\ref{prop:cosine-lower}, gives \eqref{eq:intro-automatic-schur}. This proves Theorem~\ref{thm:intro-automatic}.
\end{proof}

For an integer $q \geq 2$, let $X_q$ be the homogeneous tree of degree $q+1$, see Figure \ref{fig:homogeneous-tree}. The preceding argument does not apply to an individual $X_q$, as the following observation shows.

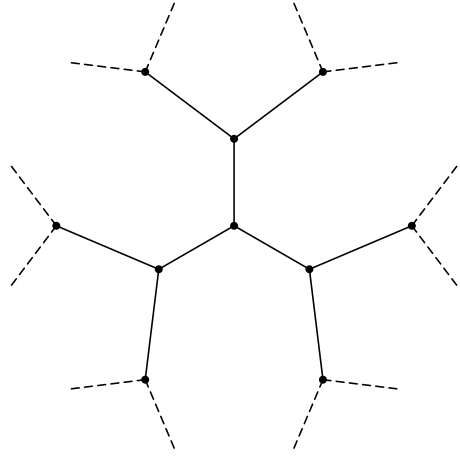
\begin{figure}[htbp]
\centering
\def\q{2}
\begin{tikzpicture}[
  vertex/.style={circle,fill=black,inner sep=0pt,minimum size=3pt},
  edge/.style={line width=0.6pt},
  continuation/.style={edge,densely dashed},
  line cap=round
]
\coordinate (O) at (0,0);

\foreach \i in {0,...,\q} {
  \pgfmathsetmacro{\a}{90+360*\i/(\q+1)}
  \coordinate (A\i) at (\a:1.15);
  \draw[edge] (O) -- (A\i);
  \node[vertex] at (A\i) {};

  \foreach \j in {1,...,\q} {
    \pgfmathsetmacro{\b}{\a+360/(\q+1)*((\j-0.5)/\q-0.5)}
    \coordinate (B\i-\j) at (\b:2.35);
    \draw[edge] (A\i) -- (B\i-\j);
    \node[vertex] at (B\i-\j) {};

    \foreach \k in {1,...,\q} {
      \pgfmathsetmacro{\c}{\b+360/((\q+1)*\q)*((\k-0.5)/\q-0.5)}
      \draw[continuation] (B\i-\j) -- (\c:3.05);
    }
  }
}

\node[vertex,label=below:{}] at (O) {};
\end{tikzpicture}
\caption{A finite portion of the homogeneous tree $X_{\q}$.
Dashed edges indicate infinite continuation.}
\label{fig:homogeneous-tree}
\end{figure}

\begin{remark} \normalfont
\label{rem:bounded-degree}
Let $X$ be a tree whose vertices have degree at most $D$, for some integer $D \geq 1$. Then $X$ does not satisfy Definition~\ref{def:multiplicity}, for any integer $h \geq 1$. Indeed, let $y$ be a vertex of $X$ and let $R \geq 1$ be an integer. The geodesic from $y$ to a vertex of the sphere $\{x \in X : \dist_X(x,y) = R\}$ is a path of length $R$ starting at $y$ without backtracking. There are at most $D(D-1)^{R-1}$ such paths, so this sphere contains at most $D(D-1)^{R-1}$ vertices. Suppose that a configuration as in Definition~\ref{def:multiplicity} exists for some integers $L,k \geq 1$. Taking $i = j = 0$ and $b = 1$ in \eqref{eq:multiplicity-distance}, we obtain $\dist_X(x_{0,a},y_{0,1}) = 2h$ for any $1 \leq a \leq k$. Since the vertices $x_{0,1},\ldots,x_{0,k}$ are pairwise distinct, they lie in the sphere of radius $2h$ centred at $y_{0,1}$. Hence $k \leq D(D-1)^{2h-1}$. In particular, for the tree $X_q$, such a configuration does not exist as soon as $k > (q+1)q^{2h-1}$.
\end{remark}

Nevertheless, the preceding argument gives a uniform-in-degree criterion.

\begin{cor}
\label{cor:uniform-degree}
Suppose that $1 <p < \infty$. For any complex function $m \co \N_0 \to \mathbb{C}$, we have
\begin{equation}
\label{eq:degree-supremum}
\norm{M_{m \circ \dist_{X_\infty}}}_{S^p_{X_\infty} \to S^p_{X_\infty}} 
= \sup_{q \geq 2} \bnorm{M_{m \circ \dist_{X_q}}}_{S^p_{X_q} \to S^p_{X_q}}.
\end{equation}
The same equality holds for completely bounded norms. Consequently, the complex function $\widetilde{m} \co \Z \to \mathbb{C}$ induces a completely bounded Fourier multiplier $M_{\widetilde m} \co \L^p(\T) \to \L^p(\T)$ if and only if the ordinary norms in the right-hand side of \eqref{eq:degree-supremum} are uniformly bounded.
\end{cor}

\begin{proof}
Recall that in a tree any two vertices are joined by a unique path \cite[Theorem 1.5.1]{Die25}. Consequently, if $T$ and $T'$ are trees and $\varphi \co T \to T'$ is an injective map sending adjacent vertices to adjacent vertices, then $\varphi$ is isometric. Indeed, the image of the path from $x$ to $y$ is a path from $\varphi(x)$ to $\varphi(y)$, hence the unique one.

Let $q \geq 2$. We claim that any finite tree $T$ whose vertices have degree at most $q+1$ admits such a map into $X_q$. We argue by induction on the number of vertices. Let $\ell$ be a leaf of $T$, let $v$ be its neighbour and let $\varphi$ be such a map defined on $T \setminus \{\ell\}$. If $w \in T \setminus \{\ell\}$ and $\varphi(w)$ is adjacent to $\varphi(v)$, then $w$ is adjacent to $v$. Otherwise the image of the path from $v$ to $w$ together with the edge joining $\varphi(v)$ and $\varphi(w)$ would form a cycle in $X_q$. Since $v$ has at most $q$ neighbours in $T \setminus \{\ell\}$, the vertex $\varphi(v)$ has a neighbour outside $\varphi(T \setminus \{\ell\})$, and we send $\ell$ to it. Exhausting $X_q$ by balls around a fixed vertex and extending at each step, the same argument gives such a map from $X_q$ into $X_\infty$. Here we use that each vertex of $X_\infty$ has infinitely many neighbours, only finitely many of which lie in the image already constructed.

Restriction along the isometric embedding $X_q \to X_\infty$ gives one inequality. Conversely, let $F$ be a finite subset of $X_\infty$. The union $T$ of the geodesics joining the points of $F$ is a finite connected subgraph of $X_\infty$, hence a finite tree. If $q+1$ is at least its maximal degree, the preceding paragraph gives an isometric embedding of $T$ into $X_q$. Thus the corner of $M_{m \circ \dist_{X_\infty}}$ indexed by $F$ is a corner of $M_{m \circ \dist_{X_q}}$. Lemma~\ref{lem:finite-corners} gives the reverse inequality, first at the scalar level and then at every matrix level. The final assertion follows from Theorem~\ref{thm:intro-automatic} and \eqref{eq:tree-upper}.
\end{proof}

\begin{lemma}
\label{lem:even-noncb}
For any $1 < p < \infty$, $p \neq 2$, there exists an even symbol $a \co \Z \to \mathbb{C}$ such that the Fourier multiplier $M_a$ is bounded but not completely bounded on $\L^p(\T)$.
\end{lemma}

\begin{proof}
Take a bounded Fourier multiplier $M_b \co \L^p(\T) \to \L^p(\T)$ which is not completely bounded, as provided by \cite[Theorem 1.5]{Arh12}. Decompose $b$ into its restrictions to the positive integers, the negative integers and the origin. Each restriction remains a bounded Fourier multiplier, since the scalar Riesz projection is bounded. The term at the origin is completely bounded. Thus at least one of the two other terms is not completely bounded. By reflection, we may assume that $b_+(n) \ov{\mathrm{def}}{=} 1_{\{n>0\}}b(n)$ has this property. Set $a(n) \ov{\mathrm{def}}{=} b_+(n)+b_+(-n)$. This is an even bounded Fourier multiplier. If the Fourier multiplier $M_a$ were completely bounded, its composition with the completely bounded projection onto the strictly positive frequencies would make $M_{b_+}$ completely bounded on $\L^p(\T)$. This is a contradiction.
\end{proof}

\begin{cor}
\label{cor:ordinary-failure}
Fix $1 < p < \infty$, $p \neq 2$, and let $X$ satisfy Definition~\ref{def:multiplicity} for some $h$. There exists an even function $a \co \Z \to \mathbb{C}$ which induces a bounded Fourier multiplier $M_a \co \L^p(\T) \to \L^p(\T)$ such that the Schur multiplier $M_{a|_{\N_0} \circ \dist_X}$ is unbounded on the Schatten space $S^p_X$. Moreover, there exists a sequence of finitely supported even symbols $(a_N)$ such that
\begin{equation}
\label{eq:finite-ordinary-failure}
\sup_N\norm{M_{a_N}}_{\L^p(\T) \to \L^p(\T)}
<\infty,
\quad
\lim_{N \to \infty} \norm{M_{a_N|_{\N_0} \circ \dist_X}}_{S^p_X \to S^p_X} 
= \infty.
\end{equation}
\end{cor}

\begin{proof}
Choose $a$ as in Lemma~\ref{lem:even-noncb}. Proposition~\ref{prop:automatic-general} implies that $M_{a|_{\N_0} \circ \dist_X}$ cannot be bounded on $S^p_X$. Define
\[
a_N(n) 
\ov{\mathrm{def}}{=} \left(1-\frac{|n|}{N+1}\right)_+a(n),\quad n \in \Z.
\]
The multiplier with the first factor as symbol is convolution by the Fej\'er kernel, hence a contraction on $\L^p(\T)$. Therefore $\norm{M_{a_N}}_{\L^p(\T) \to \L^p(\T)} \leq \norm{M_a}_{\L^p(\T) \to \L^p(\T)}$. Each $a_N$ is finitely supported and even, and $a_N(n) \to a(n)$ for any $n$. Since $a_N$ is finitely supported, $M_{a_N}$ is completely bounded. Theorem~\ref{thm-tree-input} therefore implies that $M_{a_N|_{\N_0} \circ \dist_X}$ is completely bounded for each $N$.

If these Schur norms did not tend to infinity, there would exist a subsequence along which they remain bounded. Since $a_N(n) \to a(n)$ for any $n \in \Z$, Lemma~\ref{lem:finite-corners}, applied along this subsequence, would imply that the Schur multiplier $M_{a|_{\N_0} \circ \dist_X}$ is bounded. This contradicts Proposition~\ref{prop:automatic-general}. This proves \eqref{eq:finite-ordinary-failure}.
\end{proof}

\begin{remark} \normalfont
The first part of Corollary~\ref{cor:ordinary-failure} rules out a general ordinary-norm version of \eqref{eq:tree-upper}, even for a fixed tree $X_\infty$. It does not contradict \cite[Remark 1.2]{GPPW26}, which records a conditional implication. The second part shows that restricting the proposed estimate to finitely supported symbols does not repair the failure of a uniform constant.
\end{remark}

\section{Positive definite symbols and sharp dependence on \texorpdfstring{$p$}{p}}
\label{sec:constants}

We start with the following result. Here $\M(\T)$ denotes the space of finite complex Borel measures on $\T$, equipped with the total variation norm.

\begin{lemma}
\label{lem:measure-phase}
Suppose that $1 < p < \infty$. Let $\mu$ be a finite complex Borel measure on $\T$ and suppose that $m(n) = \int_\T z^n\d\mu(z)$ for any integer $n \geq 0$. Then, on any countable tree $X$, we have
\begin{equation}
\label{eq:measure-phase}
\norm{M_{m \circ \dist_X}}_{\cb, S^p_X \to S^p_X}
\leq \kappa_p\norm{\mu}_{\M(\T)}.
\end{equation}
\end{lemma}

\begin{proof}
For a matrix $a$ with finite support, the map $\T \to S^p_X$, $z \mapsto \Delta_z^X(a)$ is continuous. For an arbitrary $a \in S^p_X$, choose matrices $a_N$ with finite support such that $\norm{a-a_N}_{S^p_X} \to 0$. Since
\[
\sup_{z \in \T}\norm{\Delta_z^X(a)-\Delta_z^X(a_N)}_{S^p_X}
\leq \kappa_p\norm{a-a_N}_{S^p_X},
\]
the map $\T \to S^p_X$, $z \mapsto \Delta_z^X(a)$ is a uniform limit of continuous maps. It is therefore continuous and bounded by $\kappa_p\norm{a}_{S^p_X}$. It is consequently Bochner integrable with respect to the measure $\mu$. Applying the continuous matrix coefficient functionals to this integral gives 
\begin{equation}
\label{equal-utile}
M_{m \circ \dist_X}(a) 
= \int_\T\Delta_z^X(a) \d\mu(z),
\end{equation}
since $m(\dist_X(x,y)) = \int_\T z^{\dist_X(x,y)} \d\mu(z)$ for any $x,y \in X$. Its norm can be estimated by
$$
\norm{M_{m \circ \dist_X}(a)}_{S^p_X}
\ov{\eqref{equal-utile}}{=}  \norm{\int_\T\Delta_z^X(a) \d\mu(z)}_{S^p_X}
\leq \int_\T\norm{\Delta_z^X(a) }_{S^p_X}\d |\mu|(z)
\ov{\eqref{eq:kappa-definition}}{\leq} \kappa_p\norm{\mu}_{\M(\T)}\norm{a}_{S^p_X}.
$$ 
The same proof applies at each matrix level.
\end{proof}

Recall that a function $a \co \Z \to \mathbb{C}$ is positive definite if all its finite Toeplitz matrices $[a(i-j)]$ are positive semidefinite, see \cite[Appendix D]{BrO08}. By Bochner's theorem \cite{Rud90}, this is equivalent to the existence of a finite positive measure $\mu$ on $\T$ with $a(n) = \int_\T z^n \d\mu(z)$. In this case $\mu(\T) = a(0)$. By \cite[Theorem 9.6 p.~86]{Pie84} and \cite[Theorem 3.26 p.~45]{ArK23}, we have
\begin{equation}
\label{eq:pd-Fourier-norm}
\norm{M_a}_{\L^p(\T) \to \L^p(\T)} 
= \norm{M_a}_{\cb,\L^p(\T) \to \L^p(\T)} 
= a(0).
\end{equation}

\begin{prop}
\label{prop:pd-upper}
Suppose that $1 < p < \infty$. Let $a \co \Z \to \mathbb{C}$ be even and positive definite, and set $m = a|_{\N_0}$. For any countable tree $X$, we have
\begin{equation}
\label{eq:pd-upper}
\norm{M_{m \circ \dist_X}}_{\cb,S^p_X \to S^p_X} 
\leq \kappa_p\norm{M_a}_{\cb,\L^p(\T) \to \L^p(\T)} 
\lesssim p^*\norm{M_a}_{\cb,\L^p(\T) \to \L^p(\T)}.
\end{equation}
\end{prop}

\begin{proof}
Apply Lemma~\ref{lem:measure-phase} to the positive representing measure $\mu$ of $a$, and use \eqref{eq:pd-Fourier-norm}.
\begin{align*}
\MoveEqLeft
\norm{M_{m \circ \dist_X}}_{\cb,S^p_X \to S^p_X}  
\ov{\eqref{eq:measure-phase}}{\leq} \kappa_p \norm{\mu}_{\M(\T)}        
\ov{\eqref{eq:pd-Fourier-norm}}{=} \kappa_p\norm{M_a}_{\cb,\L^p(\T) \to \L^p(\T)} 
\lesssim p^*\norm{M_a}_{\cb,\L^p(\T) \to \L^p(\T)}.
\end{align*}
\end{proof}

We give a lower estimate in which the symbol does not depend on $p$.

\begin{prop}
\label{prop:cosine-lower}
Suppose that $1 < p < \infty$. Let $X_2$ be the homogeneous tree of degree three and let $m(n) = \cos(\frac{\pi n}{2})$ for $n \geq 0$. Then $\norm{M_{\widetilde m}}_{\cb,\L^p(\T) \to \L^p(\T)} = 1$ and
\begin{equation}
\label{eq:cosine-lower}
\norm{M_{m \circ \dist_{X_2}}}_{S^p_{X_2} \to S^p_{X_2}}
\geq 2r_{\cb,p}-1
\geq \frac{2}{\sin\frac{\pi}{p}}-1
\geq \frac{p^*}{\pi}.
\end{equation}
\end{prop}

\begin{proof}
The even extension $\widetilde{m}$, defined by $\widetilde{m}(n) \ov{\mathrm{def}}{=} m(|n|)$, is represented by the positive probability measure $\frac{1}{2}(\delta_{\i}+\delta_{-\i})$. Indeed, for any integer $n \in \Z$,
\[
\widetilde m(n)
= m(|n|)
= \cos\left(\frac{\pi n}{2}\right)
= \frac{1}{2}\big(\i^n+(-\i)^n\big)
= \int_\T z^n \d\left(\frac{1}{2}(\delta_{\i}+\delta_{-\i})\right)(z).
\]
Hence \eqref{eq:pd-Fourier-norm} gives its Fourier multiplier norm. Choose a geodesic ray $(v_j)_{j \geq 0}$ in the tree $X_2$. For each integer $j \geq 1$, choose the neighbour $y_j$ of $v_{2j-1}$ outside the ray, and put $x_i \ov{\mathrm{def}}{=} v_{2i}$ for any integer $i \geq 1$. The unique geodesics give
\begin{equation}
\label{eq:trivalent-distances}
\dist_{X_2}(x_i,y_j) =
\begin{cases}
2(j-i),&i<j,\\
2(i-j+1),&i \geq j.
\end{cases}
\end{equation}
The configuration is illustrated in Figure~\ref{fig:trivalent-corner}. 
\begin{figure}[ht]
\centering
\begin{tikzpicture}[every node/.style={font=\small},vertex/.style={circle,fill=black,inner sep=1.5pt}]
\node[vertex,label=below:$v_0$] (v0) at (0,0) {};
\node[vertex,label=below:$v_1$] (v1) at (1.3,0) {};
\node[vertex,label=below:{$v_2 = x_1$}] (v2) at (2.6,0) {};
\node[vertex,label=below:$v_3$] (v3) at (3.9,0) {};
\node[vertex,label=below:{$v_4 = x_2$}] (v4) at (5.2,0) {};
\node[vertex,label=below:$v_5$] (v5) at (6.5,0) {};
\node[vertex,label=below:{$v_6 = x_3$}] (v6) at (7.8,0) {};
\draw (v0) -- (v1) -- (v2) -- (v3) -- (v4) -- (v5) -- (v6);
\draw[densely dotted] (v6) -- (9,0);
\node at (9.25,0) {$\cdots$};
\node[vertex,label=above:$y_1$] (y1) at (1.3,1.15) {};
\node[vertex,label=above:$y_2$] (y2) at (3.9,1.15) {};
\node[vertex,label=above:$y_3$] (y3) at (6.5,1.15) {};
\draw (v1) -- (y1);
\draw (v3) -- (y2);
\draw (v5) -- (y3);
\end{tikzpicture}
\caption{The vertices $x_i = v_{2i}$ lie on the geodesic ray, while $y_j$ is the neighbour of $v_{2j-1}$ chosen outside the ray. Only the edges relevant to the argument are displayed.}
\label{fig:trivalent-corner}
\end{figure}
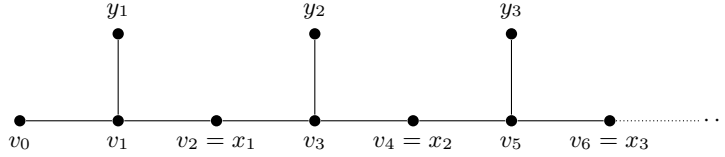
Let
\[
s(i,j) \ov{\mathrm{def}}{=} m(\dist_{X_2}(x_i,y_j)), \quad i,j \geq 1.
\]
By \eqref{eq:trivalent-distances}, we have
\[
s(i,j)
=
\begin{cases}
(-1)^{j-i},&i<j,\\
(-1)^{i-j+1},&i \geq j.
\end{cases}
\]
Now multiply row $i$ of this symbol by $(-1)^i$ and column $j$ by $(-1)^j$. If $i < j$, then
\[
(-1)^i s(i,j)(-1)^j
= (-1)^{i+j+j-i}
= (-1)^{2j}
= 1,
\]
whereas if $i \geq j$, then
\[
(-1)^i s(i,j)(-1)^j
= (-1)^{i+j+i-j+1}
= (-1)^{2i+1}
= -1.
\]
Thus the resulting symbol is
\begin{equation}
\label{eq:hilbert-corner}
h(i,j) 
\ov{\mathrm{def}}{=}
\begin{cases}
1,&i<j,\\
-1,&i \geq j.
\end{cases}
\end{equation}
More generally, let $s = [s(i,j)]$ be a Schur symbol and let $(\alpha_i)$ and $(\beta_j)$ be unimodular scalar sequences. If
\[
s'(i,j) = \alpha_i s(i,j)\beta_j,
\]
and $D_\alpha = \diag(\alpha_i)$ and $D_\beta = \diag(\beta_j)$ are the corresponding diagonal unitaries, then
\[
M_{s'}(A) = D_\alpha M_s(A)D_\beta
\]
for any matrix $A$. Since left and right multiplication by unitaries are isometries on any Schatten class, we have
\[
\norm{M_{s'}(A)}_{S^p}
= \norm{M_s(A)}_{S^p}.
\]
Consequently, we have $
\norm{M_{s'}}_{S^p \to S^p}
= \norm{M_s}_{S^p \to S^p}$ and the same argument applies at every matrix level. Now, we identify this rectangular corner with a Schatten class. Let $U_x \co \ell^2_\N \to \ell^2_{X_2}$, $e_i \mapsto e_{x_i}$ and $U_y \co \ell^2_\N \to \ell^2_{X_2}$, $e_j \mapsto e_{y_j}$. These maps are isometries. Define
\[
\Phi(A) \ov{\mathrm{def}}{=} U_x A U_y^*,
\quad
A \in S^p_\N.
\]
The operator $\Phi(A)$ is precisely the rectangular matrix supported on the rows indexed by the $x_i$ and the columns indexed by the $y_j$. Moreover,
\[
\Phi(A)^*\Phi(A)
= U_y A^*A U_y^*,
\]
so $\Phi(A)$ and $A$ have the same nonzero singular values, with the same multiplicities. Hence $
\norm{\Phi(A)}_{S^p_{X_2}}
= \norm{A}_{S^p_\N}$. By the definition of $s$, we have
\[
M_{m \circ \dist_{X_2}}(\Phi(A))
= \Phi(M_s(A)).
\]
It follows that
\[
\norm{M_{m \circ \dist_{X_2}}}_{S^p_{X_2} \to S^p_{X_2}}
\geq \norm{M_s}_{S^p_\N \to S^p_\N}
= \norm{M_h}_{S^p_\N \to S^p_\N}.
\]
Now, define the function $c \co \Z \to \mathbb{C}$ by
\[
c(n)
\ov{\mathrm{def}}{=}
\begin{cases}
1,&n < 0,\\
-1,&n \geq 0.
\end{cases}
\]
Then $h(i,j) = c(i-j)$ for any integers $i,j \geq 1$. We first show that
\begin{equation}
\label{eq:triangular-cb-equality}
\norm{M_h}_{S^p_\N \to S^p_\N}
= \norm{M_h}_{\cb,S^p_\N \to S^p_\N}.
\end{equation}
Indeed, fix integers $L,k \geq 1$ and put
\[
u_{i,a} 
\ov{\mathrm{def}}{=} 2k(i-1)+k+a,
\quad
v_{j,b} 
\ov{\mathrm{def}}{=} 2k(j-1)+b,
\]
for $1 \leq i,j \leq L$ and $1 \leq a,b \leq k$. These indices satisfy $
u_{i,a} < v_{j,b} \Longleftrightarrow  i < j$. Consequently, we have
\[
h(u_{i,a},v_{j,b}) 
= h(i,j).
\]
The indices in each of the two families are pairwise distinct. Thus, using the corresponding row and column isometries, any finite corner of the $k$-th Schatten amplification of $M_h$ can be identified isometrically with a rectangular restriction of $M_h$. We deduce that
\[
\norm{\Id_{S^p_k}\ot M_h}_{S^p_k(S^p_\N) \to S^p_k(S^p_\N)}
\leq \norm{M_h}_{S^p_\N \to S^p_\N}.
\]
Taking the supremum over $k \geq 1$ and using Lemma~\ref{lem:finite-corners} proves \eqref{eq:triangular-cb-equality}.

Moreover, after shifting the indices by one, the Toeplitz multiplier on $S^p_\N$ is identified with the restriction of $M_c^\HS$ to $S^p_{\N_0}$. As recalled after \eqref{eq:Z-transference}, restricting a Toeplitz Schur multiplier from $\Z$ to $\N_0$ does not change its norm, including at every matrix level. Hence
\[
\norm{M_h}_{\cb,S^p_\N \to S^p_\N}
= \norm{M_c^\HS}_{\cb,S^p_\Z \to S^p_\Z}
\ov{\eqref{eq:Z-transference}}{=}
\norm{M_c}_{\cb,\L^p(\T) \to \L^p(\T)}.
\]
Since $c = 1-2 1_{\N_0}$, we have $M_c = \Id-2P_+$. Consequently, we have
\begin{align*}
\MoveEqLeft
\norm{M_{m \circ \dist_{X_2}}}_{S^p_{X_2} \to S^p_{X_2}}
\geq \norm{M_h}_{S^p_\N \to S^p_\N}
\ov{\eqref{eq:triangular-cb-equality}}{=}
\norm{M_h}_{\cb,S^p_\N \to S^p_\N}\\
&= \norm{\Id-2P_+}_{\cb,\L^p(\T) \to \L^p(\T)}
\geq 2\norm{P_+}_{\cb,\L^p(\T) \to \L^p(\T)}-1\\
&= 2r_{\cb,p}-1
\geq 2\norm{P_+}_{\L^p(\T) \to \L^p(\T)}-1
\ov{\eqref{norm-P+}}{=}
\frac{2}{\sin\frac{\pi}{p}}-1.
\end{align*}
Recall that $p^* = \max\{p,p'\}$, where $p' = \frac{p}{p-1}$. If $p \geq 2$, then $p^* = p$. If $1 < p \leq 2$, then $p^* = p'$ and $\frac{\pi}{p^*} = \pi-\frac{\pi}{p}$. Therefore, using $\sin(\pi-t) = \sin(t)$, we obtain
\[
\sin\frac{\pi}{p}
= \sin\frac{\pi}{p^*}
\leq \frac{\pi}{p^*}.
\]
Since $\sin(\frac{\pi}{p}) \leq 1$, we conclude that $
\frac{2}{\sin\frac{\pi}{p}}-1
\geq \frac{1}{\sin\frac{\pi}{p}}
\geq \frac{p^*}{\pi}$. This proves \eqref{eq:cosine-lower}. Finally, the preceding argument also gives
\[
\norm{\Id-2P_+}_{\cb,\L^p(\T) \to \L^p(\T)}
\leq \norm{M_{m \circ \dist_{X_2}}}_{\cb,S^p_{X_2} \to S^p_{X_2}}
\ov{\eqref{eq:pd-upper}}{\leq} \kappa_p.
\]
It follows that $2r_{\cb,p}-1 \leq \kappa_p$. In particular, we have $r_{\cb,p} \leq \frac{\kappa_p+1}{2} \lesssim p^*$.
\end{proof}

Let $K_p$ denote the best constant in the uniform inequality 
$$
\norm{M_{m \circ \dist_X}}_{\cb,S^p_X \to S^p_X} 
\leq K_p\norm{M_{\widetilde m}}_{\cb,\L^p(\T) \to \L^p(\T)}
$$ 
over countable trees and nonzero completely bounded even symbols. Let $K_p^+$ be the corresponding constant restricted to positive definite even symbols. We have proved that
\begin{equation}
\label{eq:Kp-bounds}
\frac{2}{\sin \frac{\pi}{p}}-1 
\leq 2r_{\cb,p}-1
\leq K_p^+ 
\leq \kappa_p 
\lesssim p^*,
\quad
K_p^+ 
\leq K_p 
\leq \kappa_p^2.
\end{equation}
Thus $K_p^+ \approx p^*$. In particular, the lower order $p^*$ for the general constant is already witnessed by a positive definite symbol. The general upper order $K_p \lesssim p^*$ suggested in \cite[Remark 1.3]{GPPW26} is not established by this argument.

\section{The local amenability modulus of free balls}
\label{sec:modulus}

In this section, we compute \eqref{eq:amenability-modulus} for the balls of a free group. Fix an integer $r \geq 2$ and put $q \ov{\mathrm{def}}{=} 2r-1$. We let $
W_n \ov{\mathrm{def}}{=} \{s \in \FF_r : |s| = n\}$ and $w_n \ov{\mathrm{def}}{=} |W_n|$. Thus 
\begin{equation}
\label{wn-values}
w_0 = 1 \quad \text{and} \quad w_n = (q+1)q^{n-1}, \quad n \geq 1.
\end{equation}
Let
\begin{equation}
\label{eq:modulus-radial-sums}
\chi_n \ov{\mathrm{def}}{=} \sum_{s \in W_n}\lambda_s,
\quad \text{and} \quad
A \ov{\mathrm{def}}{=} \frac{\chi_1}{2\sqrt q}.
\end{equation}
We also write $B_N \ov{\mathrm{def}}{=} \{s \in \FF_r : |s| \leq N\}$. The operator $A$ is selfadjoint. By Kesten's spectral description \cite{Kes59}, its spectrum is $
\sigma(A) = [-1,1]$. This spectral fact is the only part of Kesten's description used in this section. The full spectral measure of $A$ will be used in Section~\ref{sec:packets}.

\begin{thm}
\label{thm:modulus}
For any integer $N \geq 0$, we have
\begin{equation}
\label{eq:exact-modulus}
1-\delta_{\FF_r}(B_N) 
= q^{-\frac{N}{2}}\left(1+\frac{q-1}{q+1}N\right).
\end{equation}
\end{thm}

\begin{proof}
For $n \geq 1$, let
\[
P_n 
\ov{\mathrm{def}}{=} \frac{1}{w_n}\chi_n
\ov{\eqref{eq:modulus-radial-sums}}{=} \frac{1}{w_n}\sum_{s \in W_n}\lambda_s
\]
be the spherical averaging operator in the left regular representation, and set
\begin{equation}
\label{def-de-rho}
P_0 \ov{\mathrm{def}}{=} \Id
\quad \text{and} \quad
\rho \ov{\mathrm{def}}{=} \frac{2\sqrt q}{q+1}.
\end{equation}
Thus $P_1 = \rho A$. The identities
\[
\chi_1^2 = (q+1)\Id+\chi_2
\quad \text{and} \quad
\chi_1\chi_n = \chi_{n+1}+q\chi_{n-1}, \quad n \geq 2,
\]
give $
P_2 = \frac{4}{q+1}A^2-\frac{1}{q}\Id$ and
\[
P_{n+1}
= \frac{2}{\sqrt q}AP_n-\frac{1}{q}P_{n-1},
\quad n \geq 2.
\]
Consequently, $P_n = p_n(A)$, with $p_0 = 1$. Setting $R_n(t) \ov{\mathrm{def}}{=} q^{\frac n2}p_n(t)$, the preceding recurrence becomes
\[
R_{n+1}(t) = 2tR_n(t)-R_{n-1}(t), \quad n \geq 2.
\]
Moreover, the first two values are
\[
R_1(t) 
= \frac{2q}{q+1}t
\quad \text{and} \quad
R_2(t) 
= \frac{4q}{q+1}t^2-1.
\]
Let $T_n$ and $U_n$ denote the Chebyshev polynomials of the first and second kinds, respectively. Since
\[
T_1(t) = t,\quad T_2(t) = 2t^2-1,\quad
U_1(t) = 2t,\quad U_2(t) = 4t^2-1,
\]
and both families satisfy the same recurrence as $R_n$, comparison of the initial values yields
\begin{equation}
\label{eq:spherical-average}
p_n(\cos\theta)
= q^{-\frac{n}{2}}\left(
\frac{q-1}{q+1}U_n(\cos\theta)
+\frac{2}{q+1}\cos(n\theta)
\right),
\quad n \geq 1,
\end{equation}
where
\[
T_n(\cos\theta) = \cos(n\theta),
\quad \text{and} \quad
U_n(\cos\theta) = \frac{\sin((n+1)\theta)}{\sin\theta}.
\]
The endpoint values are understood by continuity. Moreover, we have
\[
U_n(\cos\theta)
= \sum_{k = 0}^n \e^{\i(n-2k)\theta},
\]
and therefore $|U_n(\cos\theta)| \leq n+1$. By continuity, the same estimate holds at $\theta = 0$ and $\theta = \pi$. Hence
\[
|U_n(t)| 
\leq n+1,
\quad t \in [-1,1].
\]
Since $\sigma(A) = [-1,1]$, the continuous functional calculus and \eqref{eq:spherical-average} yield
\begin{align*}
\MoveEqLeft
\norm{P_n}_{\ell^2_{\FF_r} \to \ell^2_{\FF_r}}
= \sup_{t \in [-1,1]}|p_n(t)|
\leq q^{-\frac{n}{2}}
\left(
\frac{q-1}{q+1}(n+1)+\frac{2}{q+1}
\right)
= q^{-\frac{n}{2}}
\left(1+\frac{q-1}{q+1}n\right).
\end{align*}
On the other hand, $1 \in \sigma(A)$ and $U_n(1) = n+1$. So
\[
p_n(1)
= q^{-\frac{n}{2}}
\left(1+\frac{q-1}{q+1}n\right).
\]
Therefore
\begin{equation}
\label{eq:spherical-average-norm}
\norm{P_n}_{\ell^2_{\FF_r} \to \ell^2_{\FF_r}}
= q^{-\frac{n}{2}}
\left(1+\frac{q-1}{q+1}n\right).
\end{equation}
The same formula also holds for $n = 0$. Formula \eqref{eq:spherical-average-norm} is due to Cohen \cite{Coh82}, see also \cite{Pyt81} and \cite[Chapter 3, p.~44]{FTP83}, where it is written as $\norm{\mu_n}_{\mathrm{C}^*_\lambda(\FF_r)} = \varphi_{\frac12}(n) = \big(1+\frac{r-1}{r}n\big)(2r-1)^{-\frac{n}{2}}$ with $\mu_n = \frac{1}{w_n}1_{W_n}$. Note that $\frac{r-1}{r} = \frac{q-1}{q+1}$. We have included the short argument for the convenience of the reader.
Since $W_N$ is invariant under inversion, the operator $P_N$ is selfadjoint. Hence, for any unit vector $\xi \in \ell^2_{\FF_r}$, we have
\begin{align*}
\MoveEqLeft
\sup_{s \in B_N}|1-\la\lambda_s\xi,\xi\ra_{\ell^2_{\FF_r}}|
\geq \frac{1}{w_N}\sum_{s \in W_N}
\big(1-\Re\la\lambda_s\xi,\xi\ra_{\ell^2_{\FF_r}}\big)
= 1-\Re\la P_N\xi,\xi\ra_{\ell^2_{\FF_r}}\\
&= 1-\la P_N\xi,\xi\ra_{\ell^2_{\FF_r}}
\geq 1-\norm{P_N}_{\ell^2_{\FF_r} \to \ell^2_{\FF_r}}
\ov{\eqref{eq:spherical-average-norm}}{=}
1-q^{-\frac{N}{2}}\left(1+\frac{q-1}{q+1}N\right).
\end{align*}
Taking the infimum over all unit vectors $\xi$ gives
\[
\delta_{\FF_r}(B_N)
\geq 1-q^{-\frac N2}\left(1+\frac{q-1}{q+1}N\right).
\]

For the reverse inequality, put
\[
\phi_n \ov{\mathrm{def}}{=} q^{-\frac{n}{2}}\left(1+\frac{q-1}{q+1}n\right),
\quad n \in \N_0.
\]
If $L \geq 1$ is an integer, we define
\begin{equation}
\label{eq:critical-vectors}
f_L(s) 
\ov{\mathrm{def}}{=} q^{-\frac{|s|}{2}}1_{B_L}(s)
\quad \text{and} \quad
\xi_L 
\ov{\mathrm{def}}{=} \frac{f_L}{\norm{f_L}_{\ell^2_{\FF_r}}}.
\end{equation}
Counting the spheres gives 
\begin{equation}
\label{eq:critical-mass}
\norm{f_L}_{\ell^2_{\FF_r}}^2
= \sum_{s \in B_L}q^{-|s|}
= 1+\sum_{n = 1}^L w_n q^{-n}
\ov{\eqref{wn-values}}{=} 1+\frac{q+1}{q}L.
\end{equation}
The Cayley graph of $\FF_r$ with respect to the free generators and their inverses is the homogeneous tree of degree $q+1$, and its graph distance is $\dist(s,t) = |s^{-1}t|$. Fix integers $1 \leq n \leq L$ and $s \in W_n$. Let $e = v_0,v_1,\ldots,v_n = s$ be the geodesic from $e$ to $s$. For $t \in \FF_r$, let $v_j$ be the vertex of this geodesic closest to $t$, and put $k \ov{\mathrm{def}}{=} \dist(t,v_j)$. Then
\[
|t| = j+k
\quad \text{and} \quad 
|s^{-1}t| = n-j+k.
\]
For $k \geq 1$, there are exactly $q^k$ such vertices $t$ if $j \in \{0,n\}$, and exactly $(q-1)q^{k-1}$ if $0 < j < n$. Indeed, the vertex $v_j$ has $q+1$ neighbours, and one of them if $j \in \{0,n\}$, two of them otherwise, lie on the geodesic. Moreover, $f_L(s^{-1}t)f_L(t) = q^{-\frac{n}{2}}q^{-k}$ if $k \leq L_j \ov{\mathrm{def}}{=} L-\max\{j,n-j\}$, and $f_L(s^{-1}t)f_L(t) = 0$ otherwise. Summing first over $k$ and then over $j$, and using $L_0 = L_n = L-n$, we obtain
\begin{align*}
\MoveEqLeft
\la\lambda_sf_L,f_L\ra_{\ell^2_{\FF_r}}
= \sum_{t \in \FF_r}f_L(s^{-1}t)f_L(t)
= q^{-\frac{n}{2}}\bigg(2(1+L-n)+\sum_{j=1}^{n-1}\Big(1+\frac{q-1}{q}L_j\Big)\bigg)\\
&= q^{-\frac{n}{2}}\Big(2+(n-1)\frac{q-1}{q}\Big)L+C_n,
\end{align*}
where $C_n$ does not depend on $L$. With \eqref{eq:critical-mass}, we deduce that
\begin{equation}
\label{eq:critical-coefficients}
\lim_{L \to \infty}\la\lambda_s\xi_L,\xi_L\ra_{\ell^2_{\FF_r}}
= q^{-\frac{n}{2}}\frac{2q+(n-1)(q-1)}{q+1}
= \phi_n.
\end{equation}
This also holds for $n = 0$, since $\xi_L$ is a unit vector. Note that, with $a_0 \ov{\mathrm{def}}{=} \frac{q-1}{q+1}$, we have
\[
\frac{\phi_{n+1}}{\phi_n}
= q^{-\frac{1}{2}}\frac{1+a_0(n+1)}{1+a_0n}
\leq q^{-\frac{1}{2}}(1+a_0)
\ov{\eqref{def-de-rho}}{=} \rho
< 1.
\]
So the sequence $(\phi_n)$ is positive and decreasing. Since $\xi_L$ is nonnegative and has norm one, the Cauchy--Schwarz inequality gives
\[
0 \leq \la\lambda_s\xi_L,\xi_L\ra_{\ell^2_{\FF_r}} \leq 1
\]
for any $s \in \FF_r$. Since $B_N$ is finite, we conclude that
\[
\lim_{L \to \infty}\sup_{s \in B_N}\big|1-\la\lambda_s\xi_L,\xi_L\ra_{\ell^2_{\FF_r}}\big|
\ov{\eqref{eq:critical-coefficients}}{=} \max_{0 \leq n \leq N}(1-\phi_n)
= 1-\phi_N.
\]
Using these unit vectors in \eqref{eq:amenability-modulus} proves the reverse inequality.
\end{proof}

\begin{cor}
\label{cor:spheres}
Let $E$ be a finite and nonempty subset of $\N_0$ and put $N \ov{\mathrm{def}}{=} \max E$. Then
\[
\delta_{\FF_r}\left(\bigcup_{n \in E}W_n\right)
= \delta_{\FF_r}(W_N)
= \delta_{\FF_r}(B_N)
= 1-q^{-\frac N2}\left(1+\frac{q-1}{q+1}N\right).
\]
\end{cor}

\begin{proof}
Averaging over $W_N$ as in the proof of Theorem~\ref{thm:modulus} gives
$\delta_{\FF_r}(W_N) \geq 1-\norm{P_N}$.
Theorem~\ref{thm:modulus} and \eqref{eq:spherical-average-norm} show that the right-hand side is $\delta_{\FF_r}(B_N)$. The conclusion follows from the inclusions
$W_N \subset \bigcup_{n \in E}W_n \subset B_N$
and the monotonicity of the local amenability modulus.
\end{proof}

\begin{lemma}
\label{lem:modulus-positive-vectors}
Let $G$ be a discrete group and let $\Omega \subset G$ be finite and nonempty. Then
\[
\delta_G(\Omega)
= 1-\sup_{\norm{\xi}_{\ell^2_G} = 1}
\inf_{s \in \Omega}\Re\la\lambda_s\xi,\xi\ra_{\ell^2_G}.
\]
Moreover, the infimum defining $\delta_G(\Omega)$ may be restricted to nonnegative unit vectors.
\end{lemma}

\begin{proof}
Let $\xi \in \ell^2_G$ be a unit vector and put $\eta \ov{\mathrm{def}}{=} |\xi|$. Then $\eta$ is a nonnegative unit vector. For any $s \in G$, we have
\[
\Re\la\lambda_s\xi,\xi\ra
\leq |\la\lambda_s\xi,\xi\ra|
\leq \la\lambda_s\eta,\eta\ra
\leq 1.
\]
Since $\la\lambda_s\eta,\eta\ra \geq 0$, we obtain
\[
\sup_{s \in \Omega}|1-\la\lambda_s\eta,\eta\ra|
\leq \sup_{s \in \Omega}\big(1-\Re\la\lambda_s\xi,\xi\ra\big)
\leq \sup_{s \in \Omega}|1-\la\lambda_s\xi,\xi\ra|.
\]
Taking infima proves both assertions.
\end{proof}

\begin{prop}
\label{prop:finite-support-ratio}
Let $G$ be a discrete group and let $\Omega$ be a nonempty finite subset of $G$. For $1 < p < \infty$, we have
\[
\mathcal{R}_{G,p}(\Omega)
\leq |\Omega|^{|\frac{1}{p}-\frac{1}{2}|}.
\]
\end{prop}

\begin{proof}
Let $a \co G \to \mathbb{C}$ be a nonzero function supported in the subset $\Omega$. Regard $\overline a$ as a vector in $\ell^2_G$ and define $\omega_a(x) = \la x\delta_e,\overline a\ra$. Then $\omega_a(\lambda_s) = a(s)$ and $\norm{\omega_a} \leq \norm{a}_{\ell^2_G}$. Let $\Delta_G \co \VN(G) \to \VN(G)\otvn\VN(G)$ be the normal unital $*$-homomorphism determined by $\Delta_G(\lambda_s) = \lambda_s\ot\lambda_s$. Since $M_a = (\Id\ot\omega_a)\Delta_G$ and the slice map $\Id\ot\omega_a$ is completely bounded with norm $\norm{\omega_a}$, we obtain
\[
\norm{M_a}_{\cb,\VN(G) \to \VN(G)}
\leq \norm{a}_{\ell^2_G}
\leq |\Omega|^{\frac{1}{2}}\norm{a}_{\ell^\infty_G}.
\]
At $p = 2$, the completely bounded Fourier multiplier norm equals $\norm{a}_{\ell^\infty_G}$. Interpolation for $p \geq 2$, followed by trace duality with the conjugate symbol for $p < 2$, gives
\[
\norm{M_a}_{\cb,\L^p(\VN(G)) \to \L^p(\VN(G))}
\leq |\Omega|^{|\frac{1}{p}-\frac{1}{2}|}\norm{a}_{\ell^\infty_G}.
\]
Testing the Herz--Schur multiplier $M_a^\HS$ on matrix units gives $\norm{M_a^\HS}_{\cb,S^p_G \to S^p_G} \geq \norm{a}_{\ell^\infty_G}$. Dividing and taking the supremum proves the assertion.
\end{proof}

\begin{cor}
\label{cor:PRS-free}
For any even integer $p \geq 2$ and any $N \in \N_0$, we have
\[
\mathcal{R}_{\FF_r,p}(B_N) 
\leq \frac{q^{\frac{pN}{4}}}{1+\frac{q-1}{q+1}\cdot\frac{pN}{2}}.
\]
\end{cor}

\begin{proof}
Since $B_N^{-1} = B_N$ and $B_NB_M = B_{N+M}$ for any $N,M \in \N_0$, the product of $\frac{p}{2}$ factors alternating between $B_N$ and $B_N^{-1}$ is the ball $B_{\frac{pN}{2}}$. By Lemma~\ref{lem:modulus-positive-vectors}, the modulus used in \cite[Section 1]{PRS22}, with $\Sigma = \FF_r$, coincides with $\delta_{\FF_r}$. Hence \cite[Theorem 1.3]{PRS22}, applied with $\Sigma = \FF_r$, together with Theorem~\ref{thm:modulus}, yields
\[
\mathcal{R}_{\FF_r,p}(B_N) 
\leq \frac{1}{1-\delta_{\FF_r}(B_{\frac{pN}{2}})}
\ov{\eqref{eq:exact-modulus}}{=} \frac{q^{\frac{pN}{4}}}{1+\frac{q-1}{q+1}\cdot\frac{pN}{2}}.
\]
\end{proof}

\begin{remark} \normalfont
\label{rem:PRS-versus-trivial}
Fix an even integer $p \geq 2$ and put $\alpha_p \ov{\mathrm{def}}{=} 1-\frac{2}{p}$. Since $q \geq 3$, we have $|B_N| = 1+(q+1)\frac{q^N-1}{q-1} \leq 2q^N$. Hence Proposition~\ref{prop:finite-support-ratio} gives
\[
\mathcal{R}_{\FF_r,p}(B_N)
\leq 2^{\frac{\alpha_p}{2}}q^{\frac{\alpha_p N}{2}}.
\]
Moreover, we have
\[
\frac{p}{4}-\frac{\alpha_p}{2}
= \frac{p^2-2p+4}{4p}
> 0.
\]
Hence, for fixed $p$ and $q$, the upper bound in Corollary~\ref{cor:PRS-free} is asymptotically larger than the cardinality bound as $N \to \infty$. Since Theorem~\ref{thm:modulus} computes the modulus exactly, this discrepancy comes from the particular local transference estimate being applied, rather than from a loss in estimating the modulus.
\end{remark}

\section{Radial polynomials and quantitative Fourier--Schur separation}
\label{sec:packets}

We keep the notation of Section~\ref{sec:modulus}. Here $q = 2r-1$, $W_n$, $w_n = |W_n|$, $\chi_n$ and $A = \frac{\chi_1}{2\sqrt q}$, see \eqref{wn-values} and \eqref{eq:modulus-radial-sums}. Put $\gamma \ov{\mathrm{def}}{=} \frac{q}{q+1}$. Thus $\frac{3}{4} \leq \gamma < 1$. The identities $\chi_1^2 = (q+1)\Id+\chi_2$ and $\chi_1\chi_n = \chi_{n+1}+q\chi_{n-1}$ for $n \geq 2$ show that the radial von Neumann subalgebra of $\VN(\FF_r)$ is the abelian von Neumann algebra generated by $A$. Throughout this section, $N$ denotes a positive integer. For $1 < p < \infty$, we put $\alpha_p \ov{\mathrm{def}}{=} |1-\frac{2}{p}|$.

We use the spectral description originating in \cite{Kes59}, see also \cite[Section 2]{GPPW26}. The spectrum of $A$ is the interval $[-1,1]$, and its spectral measure for $\tau$ is the probability measure
\begin{equation}
\label{eq:Kesten-measure}
\d\nu_\gamma(x) = \frac{2\gamma}{\pi}\frac{\sqrt{1-x^2}}{1-4\gamma(1-\gamma)x^2}\d x,
\quad -1 \leq x \leq 1.
\end{equation}
Under the spectral isomorphism, $\frac{1}{\sqrt{w_n}}\chi_n$ corresponds to a real polynomial $Q_n$. Its first values and recurrence are
\begin{equation}
\label{eq:Q-recurrence}
Q_0 = 1,\quad Q_1(x) = 2\sqrt\gamma x,\quad
Q_2(x) = 4\sqrt\gamma x^2-\frac{1}{\sqrt\gamma},\quad
Q_{n+1}(x) = 2xQ_n(x)-Q_{n-1}(x),\quad n \geq 2.
\end{equation}
These polynomials are orthonormal in $\L^2(\nu_\gamma)$. A radial Fourier multiplier with profile $m$ acts by $Q_n \mapsto m(n)Q_n$. For any finite linear combination $f$ of the $Q_n$, its norm in the radial part of $\L^p(\VN(\FF_r))$ is exactly $\norm{f}_{\L^p(\nu_\gamma)}$.

For $0 < \theta < \pi$, put $x \ov{\mathrm{def}}{=} \cos\theta$. The recurrence gives
\begin{equation}
\label{eq:Q-trigonometric}
Q_n(\cos\theta) = \frac{a\sin((n+1)\theta)-b\sin((n-1)\theta)}{\sin\theta},
\quad a = \sqrt\gamma,\quad b = \frac{1-\gamma}{\sqrt\gamma},\quad n \geq 1.
\end{equation}
For example, this identity follows by comparing \eqref{eq:Q-recurrence} with the recurrence of the Chebyshev polynomials of the second kind, with $U_{-1} = 0$. At the endpoints we use the continuous polynomial values. Notice that $a-b = \frac{2\gamma-1}{\sqrt\gamma}$ is bounded below by a positive numerical constant for $\gamma \in [\frac{3}{4},1]$.

The angular form of \eqref{eq:Kesten-measure} is
\begin{equation}
\label{eq:angular-measure}
\d\nu_\gamma(\cos\theta) = w_\gamma(\theta)\d\theta,
\quad
w_\gamma(\theta) = \frac{2\gamma}{\pi}\frac{\sin^2\theta}{1-4\gamma(1-\gamma)\cos^2\theta}.
\end{equation}
Since $0 \leq 4\gamma(1-\gamma) \leq \frac{3}{4}$, there are numerical constants $c,C > 0$ such that
\begin{equation}
\label{eq:angular-comparison}
c\sin^2\theta \leq w_\gamma(\theta) 
\leq C\sin^2\theta
\end{equation}
for any $\gamma \in [\frac{3}{4},1]$ and any $\theta \in [0,\pi]$. This uniformity will make all our lower estimates independent of the free rank.

We first record a discrete Fourier estimate. In the following lemma, the distance to $2\pi\Z$ is the usual distance in $\R$.

\begin{lemma}
\label{lem:smooth-kernel}
Let $\psi \in \C_c^\infty(\R)$ and define
$F_N(t) = \sum_{k \in \Z}\psi(\tfrac{k}{N})\e^{\i kt}$ for integers $N \geq 1$. For any integers $j,L \geq 0$, there is a constant $C_{\psi,j,L}$ such that
\begin{equation}
\label{eq:smooth-kernel}
|F_N^{(j)}(t)| \leq C_{\psi,j,L}N^{j+1}
\bigl(1+N\dist(t,2\pi\Z)\bigr)^{-L},\quad t \in \R.
\end{equation}
In particular, we have $\sup_{N \geq 1} \int_0^{2\pi}|F_N(t)| \d t < \infty$.
\end{lemma}

\begin{proof}
The bound $|F_N^{(j)}(t)| \leq C N^{j+1}$ follows by summing the absolute values of its coefficients. Put $c_k \ov{\mathrm{def}}{=} (\i k)^j\psi(\frac{k}{N})$, and let $\Delta c_k \ov{\mathrm{def}}{=} c_k-c_{k-1}$. Discrete summation by parts gives
\[
(1-\e^{\i t})^L F_N^{(j)}(t) 
= \sum_{k \in \Z}(\Delta^L c)_k\e^{\i kt}.
\]
By expressing each finite difference as an integral of a derivative, we obtain $\sum_k|(\Delta^L c)_k| \leq C N^{j+1-L}$. Indeed, the $L$-th derivative of $x \mapsto (\i x)^j\psi(\frac{x}{N})$ is bounded by $CN^{j-L}$, and only $O(N)$ integer indices contribute. The constants are permitted to depend on $j,L$ and the support of $\psi$. Since $|1-\e^{\i t}|$ is comparable to $\operatorname{dist}(t,2\pi\Z)$, combining this estimate with the first bound proves \eqref{eq:smooth-kernel}. Integrating \eqref{eq:smooth-kernel} with $j = 0$ and $L = 2$ proves the last assertion.
\end{proof}

Fix for the rest of this section a nonnegative, nonzero function $\psi \in \C_c^\infty((\frac{1}{4},\frac{3}{4}))$. For an angle $0 < \vartheta < \pi$, define the two finite radial sums
\begin{equation}
\label{eq:packets}
u_N(x)
 = \sum_{k \geq 1}\psi(\tfrac{k}{N})Q_k(x),
\quad
v_{N,\vartheta}(x) 
= \sum_{k \geq 1}\psi(\tfrac{k}{N})\e^{-\i k\vartheta}Q_k(x).
\end{equation}

\begin{lemma}
\label{lem:packet-norms}
Let $1 < p < \infty$ and $0 < \vartheta < \pi$. For all sufficiently large $N$, we have
\begin{equation}
\label{eq:packet-norms}
\norm{u_N}_{\L^p(\nu_\gamma)} 
\approx_{p,\psi} N^{2-\frac{3}{p}},
\quad
\norm{v_{N,\vartheta}}_{\L^p(\nu_\gamma)} 
\approx_{p,\psi,\vartheta} N^{1-\frac{1}{p}}.
\end{equation}
The constants and the lower threshold for $N$ may be chosen uniformly for $\gamma \in [\frac{3}{4},1]$. They may also be chosen uniformly when $\psi(y)$ is replaced by $\psi(y)\e^{uy}$ with $u \in [-1,1]$.
\end{lemma}

\begin{proof}
In this proof, implicit constants may depend on $p$, $\psi$ and $\vartheta$, but not on $N$ or $\gamma$. For estimates involving $u_N$ only, the constants do not depend on $\vartheta$. Write $F_N(t) = \sum_k\psi(\frac{k}{N})\e^{\i kt}$, and put $c_\gamma(\theta) \ov{\mathrm{def}}{=} a\e^{\i\theta}-b\e^{-\i\theta}$. Formula \eqref{eq:Q-trigonometric} gives
\begin{align}
\label{eq:packet-representation}
u_N(\cos\theta)
& = \frac{c_\gamma(\theta)F_N(\theta)-c_\gamma(-\theta)F_N(-\theta)}{2\i\sin\theta},\notag\\
v_{N,\vartheta}(\cos\theta)
& = \frac{c_\gamma(\theta)F_N(\theta-\vartheta)-c_\gamma(-\theta)F_N(-\theta-\vartheta)}{2\i\sin\theta}.
\end{align}
Both numerators vanish at $\theta = 0$ and $\theta = \pi$. This cancellation must be retained when estimating the quotient near the endpoints.

We first estimate $u_N$. For any integer $L \geq 1$, we claim that
\begin{equation}
\label{eq:unmodulated-pointwise}
|u_N(\cos\theta)| \leq C_L N^2(1+N\theta)^{-L},\quad 0 \leq \theta \leq \pi.
\end{equation}
On $[0,N^{-1}]$, the derivative of the numerator in the first line of \eqref{eq:packet-representation} is $O(N^2)$, by Lemma~\ref{lem:smooth-kernel}. Since the numerator vanishes at zero and $\sin\theta$ is comparable to $\theta$, the quotient is $O(N^2)$. On $[N^{-1},\frac{\pi}{4}]$, Lemma~\ref{lem:smooth-kernel} bounds the numerator by $C_LN(1+N\theta)^{-L}$. Division by $\sin\theta$ proves \eqref{eq:unmodulated-pointwise} on this interval. On $[\frac{\pi}{4},\frac{3\pi}{4}]$, the denominator is bounded below, and both arguments of $F_N$ stay away from $2\pi\Z$; the desired bound follows again from \eqref{eq:smooth-kernel}. Finally, on $[\frac{3\pi}{4},\pi]$, the derivative of the numerator is $O_L(N^{2-L})$. Its vanishing at $\pi$ cancels the factor $\sin\theta$, which proves \eqref{eq:unmodulated-pointwise} on the remaining interval. Increasing $L$ in \eqref{eq:smooth-kernel} where needed changes only the constant.

Choose $L$ so that $Lp > 3$. We obtain
\begin{align*}
\norm{u_N}_{\L^p(\nu_\gamma)}^p
&\ov{\eqref{eq:angular-comparison}}{\lesssim}
N^{2p}\int_0^\pi(1+N\theta)^{-Lp}\theta^2\d\theta\\
& \leq C N^{2p-3}\int_0^\infty(1+s)^{-Lp}s^2\d s
\lesssim N^{2p-3}.
\end{align*}
For the lower bound, choose a sufficiently small numerical $\varepsilon > 0$. If $0 \leq \theta \leq \frac{\epsi}{N}$ and $\psi(\frac{k}{N}) \neq 0$, then \eqref{eq:Q-trigonometric} gives $Q_k(\cos\theta) \geq c k$, uniformly in $\gamma$. To check this, use $\sin((k+1)\theta) \geq (1-\epsi_0)(k+1)\theta$ and $\sin((k-1)\theta) \leq (k-1)\theta$, with $\epsi_0$ small compared with the uniform lower bound of $\frac{a-b}{a}$. The value at $\theta = 0$ follows by continuity. Since $\sum_k k\psi(\frac{k}{N}) \geq c_\psi N^2$ for large $N$, we obtain $u_N(\cos\theta) \geq c_\psi N^2$ on this interval. Consequently, we have
\[
\norm{u_N}_{\L^p(\nu_\gamma)}^p
\ov{\eqref{eq:angular-comparison}}{\gtrsim}
N^{2p}\int_0^{\frac{\varepsilon}{N}}\theta^2\d\theta
\gtrsim N^{2p-3}.
\]
This proves the first equivalence in \eqref{eq:packet-norms}.

We now consider $v_{N,\vartheta}$. For any integer $L \geq 1$, the second line of \eqref{eq:packet-representation} gives
\begin{equation}
\label{eq:modulated-pointwise}
|v_{N,\vartheta}(\cos\theta)| \leq C_{L,\vartheta}N(1+N|\theta-\vartheta|)^{-L},\quad 0 \leq \theta \leq \pi.
\end{equation}
Indeed, on a closed interval separated from zero and $\pi$, the denominator is bounded below. The only possible concentration of $F_N(\theta-\vartheta)$ occurs at $\theta = \vartheta$, while $-\theta-\vartheta$ remains a positive distance from $2\pi\Z$. Near either endpoint, both arguments stay away from $2\pi\Z$. Apply \eqref{eq:smooth-kernel} to the derivative of the numerator, with exponent $L+1$, and use its vanishing at that endpoint. The resulting bound is $O_{L,\vartheta}(N^{1-L})$, as required by \eqref{eq:modulated-pointwise} there. The coefficients $c_\gamma$ and their derivatives are uniformly bounded in $\gamma$.

Since the density $w_\gamma$ is uniformly bounded, integration of \eqref{eq:modulated-pointwise}, with $Lp > 1$, yields
$\norm{v_{N,\vartheta}}_{\L^p(\nu_\gamma)}^p \lesssim N^{p-1}$. For the reverse estimate, fix a small $\varepsilon > 0$. If $|\theta-\vartheta| \leq \frac{\varepsilon}{N}$, then
\[
\Re F_N(\theta-\vartheta) 
\geq \cos(\varepsilon)\sum_k\psi(\tfrac{k}{N}) \geq c_\psi N
\]
for sufficiently large $N$. Moreover, $|c_\gamma(\theta)| \geq a-b \geq c > 0$, whereas $F_N(-\theta-\vartheta) = O_{\vartheta,L}(N^{1-L})$. The reverse triangle inequality in \eqref{eq:packet-representation} therefore gives $|v_{N,\vartheta}(\cos\theta)| \geq c_{\psi,\vartheta}N$ on that interval. The density is bounded below there by a positive constant depending only on $\vartheta$. Thus its $p$-th norm power is at least $cN^{p-1}$.

All estimates use finitely many seminorms of $\psi$, together with positive lower bounds for $\int\psi$ and $\int y\psi(y)\d y$. These seminorms and lower bounds are uniform for the family $\psi(y)\e^{uy}$, $u \in [-1,1]$. This proves the final assertion.\end{proof}

We use the same notation $u_N$ and $v_{N,\vartheta}$ for the corresponding elements of the radial subalgebra of $\VN(\FF_r)$.

\begin{prop}
\label{prop:sharp-cutoff}
Suppose that $1 < p < \infty$. Let $0 < \vartheta < \pi$ and define $a_{N,\vartheta}(s) \ov{\mathrm{def}}{=} \e^{\i\vartheta|s|}1_{B_N}(s)$. Then
\begin{equation}
\label{eq:sharp-cutoff}
\norm{M_{a_{N,\vartheta}}^\HS}_{\cb,S^p_{\FF_r} \to S^p_{\FF_r}} 
\leq 4r_{\cb,p}\kappa_p^2,
\quad
\norm{M_{a_{N,\vartheta}}}_{\L^p(\VN(\FF_r)) \to \L^p(\VN(\FF_r))} 
\gtrsim_{p,\vartheta} N^{\alpha_p}
\end{equation}
for sufficiently large $N$, uniformly for $2 \leq r < \infty$.
\end{prop}

\begin{proof}
The even extension of the profile is
\[
\e^{\i\vartheta|n|}1_{\{|n| \leq N\}}
 = \e^{\i\vartheta n}1_{\{0 \leq n \leq N\}}
+\e^{-\i\vartheta n}1_{\{-N \leq n \leq -1\}}.
\]
Multiplication of a Fourier symbol by $\e^{\i\vartheta n}$ composes the multiplier with a translation on the torus, which is a complete isometry. Each interval projection has completely bounded norm at most $2r_{\cb,p}$. Thus the torus multiplier has completely bounded norm at most $4r_{\cb,p}$, and \eqref{eq:tree-upper} gives the first inequality in \eqref{eq:sharp-cutoff}. To use the Cayley graph with the convention \eqref{eq:herz-intro}, note that simultaneous reindexing by $s \mapsto s^{-1}$ transforms $|st^{-1}|$ into $|s^{-1}t|$.

All degrees in \eqref{eq:packets} are smaller than $N$. If $p > 2$, then $M_{a_{N,\vartheta}}v_{N,\vartheta} = u_N$, so
\[
\norm{M_{a_{N,\vartheta}}}_{\L^p(\VN(\FF_r)) \to \L^p(\VN(\FF_r))} 
\geq \frac{\norm{u_N}_p}{\norm{v_{N,\vartheta}}_p}
\ov{\eqref{eq:packet-norms}}{\gtrsim} N^{1-\frac{2}{p}}.
\]
If $1 < p < 2$, then $M_{a_{N,\vartheta}}u_N = \overline{v_{N,\vartheta}}$ in the scalar spectral model. Consequently,
\[
\norm{M_{a_{N,\vartheta}}}_{\L^p(\VN(\FF_r)) \to \L^p(\VN(\FF_r))} 
\geq \frac{\norm{v_{N,\vartheta}}_p}{\norm{u_N}_p}
\ov{\eqref{eq:packet-norms}}{\gtrsim} N^{\frac{2}{p}-1}.
\]
The uniformity in $r$ follows from Lemma~\ref{lem:packet-norms}. For $p = 2$, both multiplier norms are equal to the supremum modulus of the symbol, which is one for sufficiently large $N$. Thus the assertion also holds in this case.
\end{proof}

The first bound in \eqref{eq:sharp-cutoff} is $O((p^*)^3)$, rather than the $O((p^*)^4)$ bound obtained by estimating the phase and the ball separately. Smooth cutoffs give a further improvement.

\begin{prop}
\label{prop:smooth-cutoff}
Suppose that $1 < p < \infty$. Choose a function $\eta \in \C_c^\infty((\frac{1}{8},\frac{7}{8}))$ with $0 \leq \eta \leq 1$ and $\eta = 1$ on the interval $[\frac{1}{4},\frac{3}{4}]$. For $0 < \vartheta < \pi$, define
\begin{equation}
\label{eq:smooth-symbol}
b_{N,\vartheta}(s) 
= \e^{\i\vartheta|s|}\eta(\tfrac{|s|}{N}).
\end{equation}
There is a constant $C_\eta$, independent of $p,r,N,\vartheta$, such that
\begin{equation}
\label{eq:smooth-separation}
\norm{M_{b_{N,\vartheta}}^\HS}_{\cb,S^p_{\FF_r} \to S^p_{\FF_r}} 
\leq C_\eta\kappa_p,
\quad
\norm{M_{b_{N,\vartheta}}}_{\L^p(\VN(\FF_r)) \to \L^p(\VN(\FF_r))} 
\gtrsim_{p,\vartheta} N^{\alpha_p}
\end{equation}
for sufficiently large $N$. The constants in the second estimate are independent of $r \geq 2$.
\end{prop}

\begin{proof}
Let $\mu_N$ be the complex measure on $\T$ whose density with respect to $\mu_\T$ is
\[
K_N(z) 
= \sum_{k \in \Z}\e^{\i\vartheta k}\eta\big(\tfrac{k}{N}\big)z^{-k}.
\]
By Lemma~\ref{lem:smooth-kernel}, we have $\norm{\mu_N}_{\M(\T)} \leq C_\eta$ uniformly in $N$; hanging $\vartheta$ only translates the kernel. For any integer $n \geq 0$, its moment is $\int_\T z^n\d\mu_N(z) = \e^{\i\vartheta n}\eta(\frac{n}{N})$. Lemma~\ref{lem:measure-phase} proves the first assertion. For the second assertion, the cutoff is equal to one on the degrees used by \eqref{eq:packets}. The two tests in the proof of Proposition~\ref{prop:sharp-cutoff} therefore apply without change.

For $p = 2$, both multiplier norms are equal to the supremum modulus of the symbol, which is one for sufficiently large $N$. Thus the assertion also holds in this case.
\end{proof}

\begin{proof}[of Theorem~\ref{thm:intro-finite}]
Take $\vartheta = \frac{\pi}{2}$. The sharp-cutoff assertion is Proposition~\ref{prop:sharp-cutoff}, since $\kappa_p \lesssim p^*$ by Theorem~\ref{thm-tree-input} and $r_{\cb,p} \lesssim p^*$ by the proof of Proposition~\ref{prop:cosine-lower}. The smooth-cutoff assertion is Proposition~\ref{prop:smooth-cutoff} and $\kappa_p \lesssim p^*$. The symbol in \eqref{eq:smooth-symbol} is supported in $B_N$. In \eqref{eq:local-ratio}, use this symbol, bound the numerator below by its ordinary Fourier norm, and apply \eqref{eq:smooth-separation}. This proves \eqref{eq:intro-ratio}.\end{proof}

\begin{cor}
\label{cor:free-ratio-bounds}
Suppose that $1 < p < \infty$ and put $\alpha_p \ov{\mathrm{def}}{=} |1-\frac{2}{p}|$. Let $r \geq 2$ be an integer and put $q \ov{\mathrm{def}}{=} 2r-1$. For any integer $N \geq 0$, we have
\begin{equation}
\label{eq:free-ratio-upper}
\mathcal R_{\FF_r,p}(B_N)
\leq |B_N|^{\frac{\alpha_p}{2}}
\leq 2^{\frac{\alpha_p}{2}}q^{\frac{\alpha_p N}{2}}
\leq 2^{\frac{\alpha_p}{2}}(1+N)^{\alpha_p}
\big(1-\delta_{\FF_r}(B_N)\big)^{-\alpha_p}.
\end{equation}
Moreover, if $p \neq 2$, there exist constants $c_p > 0$ and $N_p \geq 1$, independent of $r$, such that, for any integer $N \geq N_p$,
\begin{equation}
\label{eq:free-ratio-modulus}
c_p\left(\frac{2}{\log q}\log\frac{1}{1-\delta_{\FF_r}(B_N)}\right)^{\alpha_p}
\leq \mathcal R_{\FF_r,p}(B_N)
\leq 2^{\frac{\alpha_p}{2}}(1+N)^{\alpha_p}
\big(1-\delta_{\FF_r}(B_N)\big)^{-\alpha_p}.
\end{equation}
\end{cor}

\begin{proof}
By Proposition~\ref{prop:finite-support-ratio} and the following estimate, valid since $q \geq 3$,
\[
|B_N|
= 1+(q+1)\frac{q^N-1}{q-1}
\leq 2q^N,
\]
we obtain
\[
\mathcal R_{\FF_r,p}(B_N)
\leq |B_N|^{\frac{\alpha_p}{2}}
\leq 2^{\frac{\alpha_p}{2}}q^{\frac{\alpha_p N}{2}}.
\]
Moreover, Theorem~\ref{thm:modulus} gives
\[
q^{\frac N2}
\ov{\eqref{eq:exact-modulus}}{=}
\frac{1+\frac{q-1}{q+1}N}{1-\delta_{\FF_r}(B_N)}
\leq \frac{1+N}{1-\delta_{\FF_r}(B_N)}.
\]
Raising this inequality to the power $\alpha_p$ proves \eqref{eq:free-ratio-upper}.

Now, suppose that $p \neq 2$. By \eqref{eq:exact-modulus}, we have
\[
\log\frac{1}{1-\delta_{\FF_r}(B_N)}
= \frac N2\log q-\log\left(1+\frac{q-1}{q+1}N\right)
\leq \frac N2\log q.
\]
Consequently, Theorem~\ref{thm:intro-finite} yields, for sufficiently large $N$,
\[
\mathcal R_{\FF_r,p}(B_N)
\geq c_pN^{\alpha_p}
\geq c_p\left(\frac{2}{\log q}
\log\frac{1}{1-\delta_{\FF_r}(B_N)}\right)^{\alpha_p}.
\]
This proves \eqref{eq:free-ratio-modulus}.
\end{proof}

\begin{remark} \normalfont
Theorem~\ref{thm:intro-finite} concerns modulated ball symbols and smooth modulated radial cutoffs, not the unmodulated projections $M_{1_{B_N}}$. If these projections were uniformly bounded, Abel summation would still give a bound of order $N$ for $\sum_{k = 0}^N\e^{\i k\vartheta}(M_{1_{B_k}}-M_{1_{B_{k-1}}})$, with $M_{1_{B_{-1}}} = 0$. This is compatible with the lower order $N^{\alpha_p}$, since $0 < \alpha_p < 1$. Thus Theorem~\ref{thm:intro-finite} does not settle the free ball projection problem near $p = 2$.
\end{remark}

We finish this section with damped phases. Let $z \in \mathbb{C}$ with $|z|<1$. The Hankel matrix associated with the profile $n \mapsto z^n$ in \cite[Theorem 1.2 p.~1338]{HSS10} is
\[
H
= \big[z^{j+k}-z^{j+k+2}\big]_{j,k \geq 0}
= \big[(1-z^2)z^{j+k}\big]_{j,k \geq 0}.
\]
It has rank one and trace norm $\frac{|1-z^2|}{1-|z|^2}$, and the constants $c_\pm$ of \cite[Theorem 1.2 p.~1338]{HSS10} vanish since $z^n \to 0$. For the homogeneous tree of degree $q+1$ with $q < \infty$, let $\mathcal{V}(B) = SBS^*$, where $S$ is the forward shift. The map $\mathcal{V}$ is a contraction on the trace class $S^1$. Hence
\[
\Big(1-\frac{1}{q}\Big)\norm{\Big(\Id-\frac{\mathcal{V}}{q}\Big)^{-1}H}_{S^1}
\leq \Big(1-\frac{1}{q}\Big)\sum_{k \geq 0} q^{-k}\norm{H}_{S^1}
= \norm{H}_{S^1}.
\]
Consequently, the Schur multiplier with symbol $(s,t) \mapsto z^{|st^{-1}|}$ on $\B(\ell^2_{\FF_r})$ has norm at most $\frac{|1-z^2|}{1-|z|^2}$, for any $2 \leq r < \infty$. Since Schur multipliers on $\B(\ell^2_{\FF_r})$ are automatically completely bounded with the same norm, the theorem of Bo\.{z}ejko and Fendler \cite{BoF84} yields
\begin{equation}
\label{eq:HSS-geometric}
\norm{M_{z^{|\cdot|}}}_{\cb,\VN(\FF_r) \to \VN(\FF_r)}
\leq \frac{|1-z^2|}{1-|z|^2}.
\end{equation}

\begin{prop}
\label{prop:abel-detailed}
Suppose that $1 < p < \infty$. Assume that $0 < \vartheta < \pi$. Set $z \ov{\mathrm{def}}{=} \e^{-t+\i\vartheta}$ and $a_{t,\vartheta}(s) \ov{\mathrm{def}}{=} z^{|s|}$. Then
\begin{equation}
\label{eq:abel-upper-explicit}
\norm{M_{a_{t,\vartheta}}}_{\cb,\L^p(\VN(\FF_r)) \to \L^p(\VN(\FF_r))} 
\leq \left(\frac{|1-\e^{-2t+2\i\vartheta}|}{1-\e^{-2t}}\right)^{\alpha_p},\quad t>0.
\end{equation}
For $0 < t \leq 1$,
\begin{equation}
\label{eq:abel-lower}
\norm{M_{a_{t,\vartheta}}}_{\L^p(\VN(\FF_r)) \to \L^p(\VN(\FF_r))} 
\gtrsim_{p,\vartheta} t^{-\alpha_p},
\quad
\norm{M_{a_{t,\vartheta}}^\HS}_{\cb, S^p_{\FF_r} \to S^p_{\FF_r}} 
\leq \kappa_p.
\end{equation}
All constants are independent of $2 \leq r < \infty$.
\end{prop}

\begin{proof}
At $p = 2$, the completely bounded norm of a Fourier multiplier is the supremum modulus of its symbol, here equal to one. Interpolating this fact with \eqref{eq:HSS-geometric} gives \eqref{eq:abel-upper-explicit} for $p \geq 2$. For $p < 2$, use trace duality and the conjugate symbol. The expression on the right is unchanged by conjugation. The maps at the endpoints are consistent on the group algebra, so these interpolation estimates apply to their canonical $\L^p$ extensions at each matrix level.

In the lower estimates below, implicit constants may depend on $p$ and $\vartheta$, but not on $r$ or $t$. For the lower estimate, first suppose that $t$ is small, and choose $N = \lfloor \frac{1}{t}\rfloor$, so that $\frac{1}{2} \leq tN \leq 1$. If $p > 2$, use as input
\[
f_{N,t} 
= \sum_{k \geq 1}\psi(\tfrac{k}{N})\e^{tk}\e^{-\i k\vartheta}Q_k.
\]
Its image under $M_{a_{t,\vartheta}}$ is $u_N$. Lemma~\ref{lem:packet-norms}, with the uniformly controlled profile $y \mapsto \psi(y)\e^{(tN)y}$, gives $\norm{f_{N,t}}_p \lesssim N^{1-\frac{1}{p}}$. Hence the norm ratio is at least $cN^{1-\frac{2}{p}} \gtrsim t^{-\alpha_p}$. If $p < 2$, use $u_N$ as input. Its image is the conjugate of a modulated sum with profile $y \mapsto \psi(y)\e^{-(tN)y}$. Its norm is at least $cN^{1-\frac{1}{p}}$, while $\norm{u_N}_p \lesssim N^{2-\frac{3}{p}}$. This gives $cN^{\frac{2}{p}-1} \gtrsim t^{-\alpha_p}$. At $p = 2$, the assertion is immediate. Extending the bound to the remaining compact interval of $t$ in $(0,1]$ only changes the constant, since $M_{a_{t,\vartheta}}(1) = 1$.

Finally, the sequence $n \mapsto \e^{-tn}\e^{\i\vartheta n}$, for $n \geq 0$, is the moment sequence of a rotated Poisson probability measure on $\T$. Lemma~\ref{lem:measure-phase} gives the Schur bound in \eqref{eq:abel-lower}.\end{proof}

Since the fraction in \eqref{eq:abel-upper-explicit} is $O(t^{-1})$ for $0 < t \leq1$, Proposition~\ref{prop:abel-detailed} proves Theorem~\ref{thm:intro-abel}. For $\vartheta = \frac{\pi}{2}$, the explicit upper bound becomes
\begin{equation}
\label{eq:abel-coth}
\norm{M_{a_{t,\frac{\pi}{2}}}}_{\cb,\L^p(\VN(\FF_r)) \to \L^p(\VN(\FF_r))} 
\leq (\coth t)^{\alpha_p}.
\end{equation}
The lower estimate in \eqref{eq:abel-lower} therefore has the optimal power of $t$ both for ordinary and completely bounded Fourier norms. The restriction $0 < \vartheta < \pi$ is necessary. Indeed, the function $s \mapsto \e^{-t|s|}$ is positive definite on $\FF_r$ \cite{Haa79}, so the associated Fourier multiplier is completely contractive on $\L^p(\VN(\FF_r))$. Since $s \mapsto (-1)^{|s|}$ is a character of $\FF_r$, the multiplier associated with $\vartheta = \pi$ is the composition of the preceding one with a trace-preserving $*$-automorphism of $\VN(\FF_r)$. Hence all these norms are equal to $1$ for $\vartheta \in \{0,\pi\}$.

\begin{cor}
\label{cor:smooth-two-sided}
Suppose that $1 < p < \infty$. Assume that $0 < \vartheta < \pi$. For the symbols $b_{N,\vartheta}$ of \eqref{eq:smooth-symbol}, we have
\[
\norm{M_{b_{N,\vartheta}}}_{\L^p(\VN(\FF_r)) \to \L^p(\VN(\FF_r))}
\approx_{p,\vartheta,\eta}
\norm{M_{b_{N,\vartheta}}}_{\cb,\L^p(\VN(\FF_r)) \to \L^p(\VN(\FF_r))}
\approx_{p,\vartheta,\eta} N^{\alpha_p}
\]
for sufficiently large $N$, uniformly for $r \geq 2$.
\end{cor}

\begin{proof}
Only the upper estimate remains to be proved. Put $g(u) \ov{\mathrm{def}}{=} \e^u\eta(u)$ and use the Fourier transform convention $\widehat g(v) = \int_\R g(u)\e^{-\i uv}\d u$. For any integer $n \geq 0$, Fourier inversion gives
\[
\e^{\i\vartheta n}\eta\left(\frac{n}{N}\right)
= \frac{1}{2\pi}\int_\R \widehat g(v)
\e^{(-\frac{1}{N}+\i(\vartheta+\frac{v}{N}))n} \d v.
\]
For $t > 0$ and $\beta \in \R$, write $a_{t,\beta}(s) = \e^{(-t+\i\beta)|s|}$. The proof of \eqref{eq:abel-upper-explicit} applies to any real $\beta$. Since $|1-\e^{-\frac{2}{N}+2\i\beta}| \leq 1+\e^{-\frac{2}{N}}$, we obtain
\[
\norm{M_{a_{\frac{1}{N},\beta}}}_{\cb,\L^p(\VN(\FF_r)) \to \L^p(\VN(\FF_r))}
\leq \big(\coth(\tfrac{1}{N})\big)^{\alpha_p}
\leq (2N)^{\alpha_p}.
\]
Applying the integral identity to elements of the group algebra, at every matrix level, yields
\[
\norm{M_{b_{N,\vartheta}}}_{\cb,\L^p(\VN(\FF_r)) \to \L^p(\VN(\FF_r))}
\leq \frac{\norm{\widehat g}_{\L^1(\R)}}{2\pi}(2N)^{\alpha_p}.
\]
The lower estimate follows from Proposition~\ref{prop:smooth-cutoff}.
\end{proof}

We end this section with radial symbols supported in balls. Denote by $\mathcal{R}^{\mathrm{rad}}_{\FF_r,p}(B_N)$ the supremum in \eqref{eq:local-ratio} restricted to nonzero radial symbols supported in $B_N$, that is, to symbols of the form $a(s) = m(|s|)$.

\begin{prop}
\label{prop:radial-polynomial}
Suppose that $1 < p < \infty$ and let $N \geq 0$ be an integer. Let $m \co \N_0 \to \mathbb{C}$ be a complex function supported in $\{0,\ldots,N\}$ and put $a(s) \ov{\mathrm{def}}{=} m(|s|)$ for any $s \in \FF_r$. Then
\begin{equation}
\label{eq:radial-polynomial}
\norm{M_a}_{\cb,\L^p(\VN(\FF_r)) \to \L^p(\VN(\FF_r))}
\leq 2^{\frac{\alpha_p}{2}}(N+2)^{\frac{3\alpha_p}{2}}\norm{m}_{\ell^\infty_{\N_0}}.
\end{equation}
Consequently, we have $\mathcal{R}^{\mathrm{rad}}_{\FF_r,p}(B_N) \leq 2^{\frac{\alpha_p}{2}}(N+2)^{\frac{3\alpha_p}{2}}$. Moreover, if $p \neq 2$, then $\mathcal{R}^{\mathrm{rad}}_{\FF_r,p}(B_N) \geq c_pN^{\alpha_p}$ for any integer $N \geq N_p$, with $c_p$ and $N_p$ as in Theorem~\ref{thm:intro-finite}.
\end{prop}

\begin{proof}
Put $h(n) \ov{\mathrm{def}}{=} m(n)-m(n+2)$ for $n \geq 0$ and consider the Hankel matrix $H = [h(j+k)]_{j,k \geq 0}$ of \cite[Theorem 1.2 p.~1338]{HSS10}. Since $m$ is supported in $\{0,\ldots,N\}$, we have $h(n) = 0$ for $n > N$ and $|h(n)| \leq 2\norm{m}_{\ell^\infty_{\N_0}}$. Thus $H$ is supported in $\{0,\ldots,N\}^2$ and has rank at most $N+1$. Since the antidiagonal $\{j+k = n\}$ of this square contains $n+1$ entries, we obtain
\[
\norm{H}_{S^2}^2
= \sum_{n = 0}^N (n+1)|h(n)|^2
\leq 2(N+1)(N+2)\norm{m}_{\ell^\infty_{\N_0}}^2.
\]
By the Cauchy--Schwarz inequality for the singular values of $H$, we deduce that
\[
\norm{H}_{S^1}
\leq (N+1)^{\frac{1}{2}}\norm{H}_{S^2}
\leq \sqrt{2}(N+2)^{\frac{3}{2}}\norm{m}_{\ell^\infty_{\N_0}}.
\]
The constants $c_\pm$ of \cite[Theorem 1.2 p.~1338]{HSS10} vanish since $m$ has finite support. For the homogeneous tree of degree $q+1$, the formula of \cite[Theorem 1.2 p.~1338]{HSS10} gives
\[
\norm{M_a^\HS}_{\B(\ell^2_{\FF_r}) \to \B(\ell^2_{\FF_r})}
=
\left(1-\frac{1}{q}\right)
\norm{\left(\Id-\frac{\mathcal V}{q}\right)^{-1}H}_{S^1},
\]
where $\mathcal V(B) = SBS^*$ and $S$ is the forward shift on $\ell^2_{\N_0}$. Since $\mathcal V$ is contractive on $S^1$, we obtain
\[
\begin{aligned}
\left(1-\frac{1}{q}\right)
\norm{\left(\Id-\frac{\mathcal V}{q}\right)^{-1}H}_{S^1}
&\leq
\left(1-\frac{1}{q}\right)
\sum_{k \geq 0}q^{-k}\norm{H}_{S^1}\\
&= \norm{H}_{S^1}.
\end{aligned}
\]
Consequently, the Schur multiplier with symbol $(s,t) \mapsto m(|st^{-1}|)$ has norm at most $\norm{H}_{S^1}$ on $\B(\ell^2_{\FF_r})$, uniformly for $r \geq 2$. By the theorem of Bo\.{z}ejko and Fendler \cite{BoF84}, we deduce that

\[
\norm{M_a}_{\cb,\VN(\FF_r) \to \VN(\FF_r)}
\leq \sqrt{2}(N+2)^{\frac{3}{2}}\norm{m}_{\ell^\infty_{\N_0}}.
\]
At $p = 2$, the completely bounded norm of $M_a$ is $\norm{a}_{\ell^\infty_{\FF_r}} = \norm{m}_{\ell^\infty_{\N_0}}$. Interpolation for $p \geq 2$, followed by trace duality with the conjugate symbol for $p < 2$, gives \eqref{eq:radial-polynomial}. Testing $M_a^\HS$ on matrix units gives
\[
\norm{M_a^\HS}_{\cb,S^p_{\FF_r} \to S^p_{\FF_r}} 
\geq \norm{a}_{\ell^\infty_{\FF_r}},
\]
and the upper estimate for $\mathcal{R}^{\mathrm{rad}}_{\FF_r,p}(B_N)$ follows. The lower estimate follows from Theorem~\ref{thm:intro-finite}, since the symbols $b_N$ appearing there are radial and supported in $B_N$.
\end{proof}

\begin{remark} \normalfont
\label{rem:sharp-cutoff-log}
For the sharp cutoffs of Proposition~\ref{prop:sharp-cutoff}, the same argument gives an upper estimate which matches the lower estimate up to a logarithmic factor. Indeed, let $N \geq 2$ and $m(n) = \e^{\i\vartheta n}1_{\{n \leq N\}}$. With $D = \diag(\e^{\i\vartheta j})_{j \geq 0}$, $T_M = [1_{\{j+k \leq M\}}]_{j,k \geq 0}$ and $J_n = [1_{\{j+k = n\}}]_{j,k \geq 0}$, we have
\[
H 
= (1-\e^{2\i\vartheta})DT_{N-2}D+\e^{\i\vartheta(N-1)}J_{N-1}+\e^{\i\vartheta N}J_N.
\]
Reversing the order of the columns of $T_M$ gives the upper triangular matrix $U_{M+1}$ of size $M+1$ whose entries on and above the diagonal are equal to one. The block decomposition 
\[
U_{2n} = \begin{bmatrix} U_n & \mathbf{1}_n \\ 0 & U_n \end{bmatrix},
\]
where $\mathbf{1}_n$ is the $n \times n$ matrix of ones, gives $\norm{U_{2n}}_{S^1} \leq 2\norm{U_n}_{S^1}+n$. Hence $\norm{U_{2^\ell}}_{S^1} \leq (\ell+2)2^{\ell-1}$. Since $U_n$ is a corner of $U_{2^\ell}$ when $2^{\ell-1} < n \leq 2^\ell$, we obtain $\norm{U_n}_{S^1} \leq n(\log_2 n+3)$. Since $\norm{J_n}_{S^1} = n+1$, it follows that $\norm{H}_{S^1} \leq 2(N+1)(\log_2(N+1)+4)$. As in the proof of Proposition~\ref{prop:radial-polynomial}, we conclude that
\[
\norm{M_{a_{N,\vartheta}}}_{\cb,\L^p(\VN(\FF_r)) \to \L^p(\VN(\FF_r))}
\leq \big(2(N+1)(\log_2(N+1)+4)\big)^{\alpha_p}.
\]
Thus the lower estimate of Proposition~\ref{prop:sharp-cutoff} is optimal up to a logarithmic factor.
\end{remark}

\paragraph{AI statement.}
The author acknowledges the use of AI tools for mathematical discussions, language polishing and feedback during the development and preparation of this manuscript. 

\paragraph{Acknowledgment}
The author thanks Adri\'an Gonz\'ales-P\'erez for sharing \cite{GPPW26} with him.

\paragraph{Competing interests}
The author declares that he has no competing interests.

\paragraph{Data availability}
No external data sets were used. The mathematical arguments are analytic and do not rely on numerical computations.

{\footnotesize
\vspace{0.2cm}
\noindent C\'edric Arhancet\\
\noindent 6 rue Didier Daurat, 81000 Albi, France\\
URL: \href{http://sites.google.com/site/cedricarhancet}{https://sites.google.com/site/cedricarhancet}\\
cedric.arhancet@protonmail.com\\
ORCID: 0000-0002-5179-6972
}


\small
\begin{thebibliography}{GPPW26}

\bibitem[Arh12]{Arh12}
C. Arhancet.
\newblock Unconditionality, Fourier multipliers and Schur multipliers.
\newblock Colloq. Math. 127 (2012), no. 1, 17--37.

\bibitem[Arh26]{Arh26}
C. Arhancet.
\newblock Optimal square size for separating the operator and completely bounded norms of Schur multipliers on $S^4$.
\newblock Preprint, arXiv:2609.29643.

\bibitem[ArK22]{ArK22}
C. Arhancet and C. Kriegler.
\newblock Decomposable Fourier multipliers and an operator-algebraic characterization of amenability.
\newblock Preprint, arXiv:2205.13823, 2022.

\bibitem[ArK23]{ArK23}
C. Arhancet and C. Kriegler.
\newblock Projections, multipliers and decomposable maps on noncommutative $L^p$-spaces.
\newblock M\'em. Soc. Math. Fr. (N.S.), 177, 2023.

\bibitem[BoF84]{BoF84}
M. Bo\.{z}ejko and G. Fendler.
\newblock Herz--Schur multipliers and completely bounded multipliers of the Fourier algebra of a locally compact group.
\newblock Boll. Un. Mat. Ital. A (6) 3 (1984), no. 2, 297--302.

\bibitem[BrO08]{BrO08}
N. P. Brown and N. Ozawa.
\newblock $C^*$-algebras and finite-dimensional approximations.
\newblock Graduate Studies in Mathematics, 88. American Mathematical Society, Providence, RI, 2008.

\bibitem[CaS15]{CaS15}
M. Caspers and M. de la Salle.
\newblock Schur and Fourier multipliers of an amenable group acting on non-commutative $\L^p$-spaces.
\newblock Trans. Amer. Math. Soc. 367 (2015), no. 10, 6997--7013.


\bibitem[Coh82]{Coh82}
J. M. Cohen.
\newblock Operator norms on free groups.
\newblock Boll. Un. Mat. Ital. B (6) 1 (1982), no. 3, 1055--1065.

\bibitem[Die25]{Die25}
R. Diestel.
\newblock Graph Theory.
\newblock Sixth edition. Graduate Texts in Mathematics, 173. Springer, Berlin--Heidelberg, 2025.

\bibitem[FTP83]{FTP83}
A. Fig\`a-Talamanca and M. A. Picardello.
\newblock Harmonic analysis on free groups.
\newblock Lecture Notes in Pure and Applied Mathematics, 87. Marcel Dekker, Inc., New York, 1983.

\bibitem[GPPW26]{GPPW26}
A. M. Gonz\'alez-P\'erez, J. Parcet, J. P\'erez Garc\'ia and S. Wang.
\newblock Radial Schur multipliers on trees and the failure of Fourier--Schur transference.
\newblock Manuscript, 2026, 28 pp.

\bibitem[HSS10]{HSS10}
U. Haagerup, T. Steenstrup and R. Szwarc.
\newblock Schur multipliers and spherical functions on homogeneous trees.
\newblock Internat. J. Math. 21 (2010), no. 10, 1337--1382.

\bibitem[Haa79]{Haa79}
U. Haagerup.
\newblock An example of a nonnuclear $C^*$-algebra, which has the metric approximation property.
\newblock Invent. Math. 50 (1978/79), no. 3, 279--293.


\bibitem[HoV00]{HoV00}
B. Hollenbeck and I. E. Verbitsky.
\newblock Best constants for the Riesz projection.
\newblock J. Funct. Anal. 175 (2000), no. 2, 370--392.

\bibitem[HST26]{HST26}
J. Huang, F. Sukochev and A. Tomskova.
\newblock Schur multipliers with unequal operator and completely bounded norms on $\S^p$, $1 < p \neq 2 < \infty$.
\newblock Preprint, \href{https://arxiv.org/abs/2608.20933}{arXiv:2608.20933}.

\bibitem[Kes59]{Kes59}
H. Kesten.
\newblock Symmetric random walks on groups.
\newblock Trans. Amer. Math. Soc. 92 (1959), no. 2, 336--354.

\bibitem[Lar71]{Lar71}
R. Larsen.
\newblock An introduction to the theory of multipliers.
\newblock Springer-Verlag, 1971.


\bibitem[NeR11]{NeR11}
S. Neuwirth and \'E. Ricard.
\newblock Transfer of Fourier multipliers into Schur multipliers and sumsets in a discrete group.
\newblock Canad. J. Math. 63 (2011), no. 5, 1161--1187.

\bibitem[PRS22]{PRS22}
J. Parcet, \'E. Ricard and M. de la Salle.
\newblock Fourier multipliers in $\mathrm{SL}_n(\R)$.
\newblock Duke Math. J. 171 (2022), no. 6, 1235--1297.

\bibitem[Pie84]{Pie84}
J.-P. Pier.
\newblock Amenable locally compact groups.
\newblock Pure and Applied Mathematics (New York). A Wiley-Interscience Publication. John Wiley \& Sons, Inc., New York, 1984.


\bibitem[Pis98]{Pis98}
G. Pisier.
\newblock Non-commutative vector valued $\L^p$-spaces and completely $p$-summing maps.
\newblock Ast\'erisque 247, 1998.

\bibitem[Pyt81]{Pyt81}
T. Pytlik.
\newblock Radial functions on free groups and a decomposition of the regular representation into irreducible components.
\newblock J. Reine Angew. Math. 326 (1981), 124--135.



\bibitem[Rud90]{Rud90}
W. Rudin.
\newblock Fourier analysis on groups.
\newblock Wiley Classics Library, John Wiley \& Sons, Inc., New York, 1990.

\bibitem[Ser80]{Ser80}
J.-P. Serre.
\newblock Trees.
\newblock Translated from the French by John Stillwell.
Springer-Verlag, Berlin--New York, 1980.

\bibitem[Spr04]{Spr04}
N. Spronk.
\newblock Measurable Schur multipliers and completely bounded multipliers of the Fourier algebras.
\newblock Proc. London Math. Soc. (3) 89 (2004), no. 1, 161--192.

\end{thebibliography}
\end{document}